\documentclass[a4paper, 13pt]{article}

\usepackage[english]{babel}

\usepackage[T1]{fontenc}%allows one to use pipe sign, less than or equal, etc freely in tex file
\usepackage{lmodern}   
\usepackage{xcolor}%adds colors

\usepackage{amsmath,amsthm,amssymb,amsfonts}%AMS PACKAGE
\usepackage{mathtools}
\usepackage{tensor}%adds \tensor and \indices
\usepackage{tikz-cd}%adds diagrams
\usepackage{esint}%\fint
\usepackage{mathrsfs}%gives fancy mathscr
\usepackage{thmtools}
\usepackage{slashed}
\usepackage{mathabx}

\usepackage{bbm}
\usepackage{microtype}
\usepackage{accents}

\usepackage{enumitem}%enumerate package
\usepackage[toc,page]{appendix}%adds appendix
\usepackage{imakeidx}%allows one to maked indices in articles

\usepackage{hyperref}

\usepackage{emptypage}%removes numbers from empty pages
\usepackage[a4paper,	bindingoffset=0in,		left=1.2in,		right=1.2in,	top=1in,	bottom=1in,	footskip=.25in]{geometry} %Margin
\usepackage[backend=biber,style=alphabetic,sorting=ynt,sortcites]{biblatex}%References
\allowdisplaybreaks[1]

\theoremstyle{plain}
\newtheorem{theorem}{Theorem}[section]
\newtheorem*{theorem*}{Theorem}
\newtheorem{thmx}{Theorem}

\newtheorem{cor}[theorem]{Corollary}
\newtheorem{lem}[theorem]{Lemma}
\newtheorem{prop}[theorem]{Proposition}
\theoremstyle{definition}
\newtheorem{ex}[theorem]{Example}

\newtheorem{dfn}[theorem]{Definition}
\newtheorem{rem}[theorem]{Remark}
\newtheorem{rems}[theorem]{Remarks}

\theoremstyle{remark}

\newcommand{\RNum}[1]{\uppercase\expandafter{\romannumeral #1\relax}}

\makeatletter
\providecommand*{\twoheadrightarrowfill@}{%
  \arrowfill@\relbar\relbar\twoheadrightarrow
}
\providecommand*{\twoheadleftarrowfill@}{%
  \arrowfill@\twoheadleftarrow\relbar\relbar
}
\providecommand*{\xtwoheadrightarrow}[2][]{%
  \ext@arrow 0579\twoheadrightarrowfill@{#1}{#2}%
}
\providecommand*{\xtwoheadleftarrow}[2][]{%
  \ext@arrow 5097\twoheadleftarrowfill@{#1}{#2}%
}
\newcommand\setItemnumber[1]{\setcounter{enum\romannumeral\@enumdepth}{\numexpr#1-1\relax}}
\makeatother

\newcommand\norm[1]{\left\lVert#1\right\rVert}

\DeclareMathOperator\End{End}

\DeclareMathOperator\dnc{DNC}

\DeclareMathOperator\spec{Spec}

\DeclareMathOperator\id{Id}

\newcommand{\R}{\mathbb{R}}

\newcommand{\N}{\mathbb{N}}
\newcommand{\C}{\mathbb{C}}

\newcommand{\cA}{\mathcal{A}}

\newcommand{\cD}{\mathcal{D}}

\newcommand{\cK}{\mathcal{K}}

\newcommand{\cS}{\mathcal{S}}

\newcommand{\cX}{\mathcal{X}}

\AtEveryBibitem{\clearfield{url}}
\newcommand{\mathbblD}{\Delta\!\!\!\!\Delta}

\DeclareMathOperator{\Der}{\mathrm{Diff}_{\cD}}

\begin{document}
\title{Differential operators on $C^*$-algebras and applications to smooth functional calculus and Schwartz functions on the tangent groupoid}
\author{Omar Mohsen}
\date{}	
\maketitle
% 19K35 Kasparov theory ($KK$-theory)
% 19K56 Index theory
% 22A22 Topological groupoids (including differentiable and Lie groupoids)
% 22A25 Representations of general topological groups and semigroups
% 22E25 Nilpotent and solvable Lie groups
% 46L87 Noncommutative differential geometry
% 46L45 Decomposition theory for $C^*$-algebras
% 46L80 $K$-theory and operator algebras (including cyclic theory)
% 46L87 Noncommutative differential geometry
% 47G30 Pseudodifferential operators	
% 53C12 Foliations (differential geometric aspects)
% 53C29 Issues of holonomy
% 58J22 Exotic index theories
% 58J40 Pseudodifferential and Fourier integral operators on manifolds
% 53R30 Foliations; geometric theory
% 58B34 Noncommutative geometry (�  la Connes)
% 58H05 Pseudogroups and differentiable groupoids 
% 93B18 Linearizations
\begin{abstract}
We introduce the notion of a differential operator on $C^*$-algebras. This is a noncommutative analogue of a differential operator on a smooth manifold. 
We show that the common closed domain of all differential operators is closed under smooth functional calculus. As a corollary, we show that Schwartz functions on Connes tangent groupoid are closed under smooth functional calculus.
%\blfootnote{AMS subject classification: ,~ Secondary: Keywords:}
\end{abstract}
\setcounter{tocdepth}{2} %doesn't display subsections in TOC 
\tableofcontents
\section*{Introduction}
Let $M$ be a compact smooth manifold. In \cite[Chapter 2.5]{ConnesBook}, Connes introduced the tangent groupoid $$\mathbb{T}M:=M\times M\times \R_+^\times \sqcup TM\times \{0\}.$$ 
The space $\mathbb{T}M$ is equipped with a smooth manifold structure which is a special case of the deformation to the normal cone construction, see \cite[Chapter 5]{FultonBookIntersection}. 
The space $C^\infty_c(\mathbb{T}M)$ is equipped with a convolution law making it a $*$-algebra. 
The convolution law is a deformation of the usual convolution law on $M\times M$ which is used in Schwartz kernels of operators acting on $L^2M$ and the convolution law on $T_xM$ for $x\in M$ which comes from the commutative group structure. 
As Connes shows, the space $\mathbb{T}M$ very naturally gives a canonical deformation from a differential operator to its principal symbol. 
Connes used this deformation to give a conceptual proof of the Atiyah-Singer index theorem \cite{AtiyahSingerI}.

The tangent groupoid has been studied and generalised by many authors. In this article, we will be interested in the space of Schwartz functions $\cS(\mathbb{T}M)$ on $\mathbb{T}M$ introduced by Carrillo-Rouse \cite{PaoloSchwartzAlgebraTangent}. 
The $*$-algebra $C^\infty_c(\mathbb{T}M)$ can be naturally completed into a $C^*$-algebra $C^*\mathbb{T}M$. 
In \cite{PaoloSchwartzAlgebraTangent}, Carrillo-Rouse proves that $\cS(\mathbb{T}M)$ is a $*$-subalgebra of $C^*\mathbb{T}M$. 
In this article, we improve his result by showing that $\cS(\mathbb{T}M)$ is closed under smooth functional calculus. 
\begin{thmx}\label{thm:intro:Schwartz_closed_holo_Connes}
\begin{enumerate}
  \item If $a\in \cS(\mathbb{T}M)$ is a Schwartz function and $a$ is normal ($a^*a=aa^*$) and $f$ is a smooth function on some open neighbourhood of $\mathrm{spec}(a)$ with $f(0)=0$, then $f(a)\in \cS(\mathbb{T}M)$.
\item If $a\in \cS(\mathbb{T}M)$ is a Schwartz function and $f$ is a holomorphic function on some open neighbourhood of $\mathrm{spec}(a)$  with $f(0)=0$, then $f(a)\in \cS(\mathbb{T}M)$.
\end{enumerate}
% The algebra $\cS(\mathbb{T}M)$ is closed under holomorphic calculus. More precisely if $a\in \cS(\mathbb{T}M)\subseteq C^*\mathbb{T}M$ 
\end{thmx}
Of course $f(a)$ in Theorem \ref{thm:intro:Schwartz_closed_holo_Connes} means the functional calculus in $C^*\mathbb{T}M$ applied to $a$. 
The main difficulty in Theorem \ref{thm:intro:Schwartz_closed_holo_Connes} is the derivability of $f(a)$ in the $\R_+$ direction of $\mathbb{T}M$ because as $t\to 0^+$, the structure of $C^*\mathbb{T}M$ changes from compact operators on $L^2M$ to $C_0(T^*M)$. 
We remark that we suppose $f(0)=0$ because $C^*\mathbb{T}M$ isn't unital.

To prove Theorem \ref{thm:intro:Schwartz_closed_holo_Connes} we introduce differential operators on an arbitrary $C^*$-algebra $A$. Recall that a derivation on $A$ is a map $\delta:\cA\to A$ defined on a $*$-subalgebra $\cA\subseteq A$ such that 
$$\delta(ab)=\delta(a)b+a\delta(b),\quad \forall a,b\in \cA.$$ Derivations usually come from the derivative of an $\R$ action on $A$. If $A=C(N)$ and $\cA=C^\infty(N)$ where $N$ is a smooth manifold $N$, then a derivation is a vector field on $N$. 
Therefore, one can think of derivations as differential operators of order $1$ on $A$ with domain $\cA$. We introduce the notion of a differential operator of higher order on an arbitrary $C^*$-algebra $A$. 
Our differential operators coincide with differential operators on $N$ if $A=C(N)$ and $\cA=C^\infty(N)$. We show that the derivative in the direction of $\R_+$ is a differential operator on $C^*\mathbb{T}M$ of order $2$. 
The main theorem of this article is the following, see Definition  \ref{dfn:cDcA} for the precise definition of $\cD(A)\subseteq A$. We remark that $\cD(A)$ is a $*$-subalgebra of $A$ which contains $\cA$.
\begin{thmx}\label{thm:intro:higher_der}
Let $\cA\subseteq A$ be a dense $*$-subalgebra and $\cD(\cA)$ be the intersection of the domains of the closures of all differential operators defined on $\cA$. Then \begin{enumerate}
  \item If $n\in \N$, $a\in M_n(\cD(\cA))$ is normal ($a^*a=aa^*$) and $f$ is a smooth function defined on some open neighbourhood of $\mathrm{spec}(a)$ with $f(0)=0$, then $f(a)\in  M_n(\cD(\cA))$.
  \item If $n\in \N$, $a\in M_n(\cD(\cA))$ and $f$ is a holomorphic function defined on some open neighbourhood of $\mathrm{spec}(a)$ with $f(0)=0$, then $f(a)\in M_n(\cD(\cA))$.
\end{enumerate}
Here $\mathrm{spec}(a)$ is the spectrum of $a$ as an element of $M_n(A)$, the $C^*$-algebra of $n\times n$ matrices with coefficients in $A$.
\end{thmx}
The methods we use to prove \ref{thm:intro:higher_der} are based on arguments by Bratteli and Robinson \cite{UnboundedDerivations1,UnboundedDerivations2} and Pedersen \cite{PedersenOpDiffFun}. We obtain Theorem \ref{thm:intro:Schwartz_closed_holo_Connes} as a corollary of Theorem \ref{thm:intro:higher_der} because $\cD(\cS(TM))=\cS(TM)$.

The main interest in subalgebras closed under holomorphic calculus in the theory of operator algebras is that they give isomorphisms in $K$-theory, thus allowing tools from cyclic cohomology to be used. 
Cyclic cocycles are rarely defined on the whole $C^*$-algebra and only defined on a subalgebra, see for example \cite{ConnesCyclicFoliation}

We remark that our approach to Schwartz functions is different from that of Carrillo-Rouse, and our algebra differs slightly from his algebra, see Remark \ref{rem:diff_dfns_schwartz}. His approach is based on using local coordinates for $\mathbb{T}M$ and demanding growth conditions in such coordinates. 
He then shows that this is independent of local coordinates. Our approach is more global in nature. Recall that Schwartz functions on a vector space are smooth functions $f$ such that $D(f)$ is bounded whenever $D$ is a differential operator with polynomial coefficients. Analogously, on the space $\mathbb{T}M$, we define differential operators which play the role of differential operators with polynomial coefficients. We then define Schwartz functions to be functions whose derivative by such operators is bounded.
\paragraph*{Acknowledgments}
We thank the referee for his valuable remarks which helped immensely with the exposition.
\paragraph{Organization of the article.} 
The article is organized as follows.
\begin{itemize}
\item In Section \ref{sec:higher_der}, we define differential operators on $C^*$-algebras and prove Theorem \ref{thm:intro:higher_der}.
\item In Section \ref{sec:SchwartzConnes}, we give a quick introduction to Connes's tangent groupoid. We then prove Theorem \ref{thm:intro:Schwartz_closed_holo_Connes}. 
\end{itemize}
%we prove that the algebra of Schwartz functions is closed under holomorphic calculus. To our knowledge this result is new. We rely in this section on a formula that was remarked by Haj Saeedi Sadegh and Higson \cite[Page 4]{HigsonHaj}.

\section{Differential operators on \texorpdfstring{$C^*$}{Cstar}-algebras}\label{sec:higher_der}
%A differential operator is a $\C$-linear map $$\delta:\cA\to \cA$$ which satisfies $$\delta(ab)=\delta(a)b+a\delta(b),\quad \forall a,b\in \cA.$$
%We remark that we don't assume that $\delta(a^*)=-\delta(a)$, only that $\cA$ is closed under adjoint.
Throughout the article $\N$ stands for $\{1,2,\cdots\}$, $\R_+^\times$ for $\{x\in \R:x>0\}$ and $\R_+$ for $\{x\in \R:x\geq 0\}$.

\begin{ex} Let $$\mathbb{H}:=\R^3,\quad (x,y,z)\cdot (x',y',z')=(x+x',y+y',z+z'+xy'-yx')$$ be the Heisenberg group and   \begin{equation}\begin{aligned}
  \delta_{x},\delta_{y},\delta_{z}:C^\infty_c(\mathbb{H})\to C^\infty_c(\mathbb{H}),\quad 
  \delta_x(f)=xf,\quad \delta_y(f)=yf,\quad \delta_z(f)=zf.
\end{aligned}\end{equation}
The maps $\delta_x$ and $\delta_y$ satisfy the Leibniz rule \begin{align*}
\delta_{x}(f\star g)=\delta_x(f)\star g+f\star \delta_x(g),\quad \delta_{y}(f\star g)=\delta_y(f)\star g+f\star \delta_y(g)
\end{align*}
but $\delta_z$ satisfies \begin{align}\label{eqn:Heisenberg}
\delta_{z}(f\star g)=\delta_z(f)\star g+f\star \delta_z(g)+\delta_{x}(f)\star\delta_{y}(g)-\delta_{y}(f)\star\delta_{x}(g),
\end{align}
where $\star$ denotes the convolution.
We define differential operators on noncommutative algebras to be maps which satisfy a Leibniz rule as in \eqref{eqn:Heisenberg}.% For many applications, one doesn't need to look at all linear maps but only a subspace.
% We enlarge the theory by also allowing differential operators of order $0$ which correspond to differential operators of order $0$.
\end{ex}
\begin{dfn}\label{dfn:differential operators}
  Let $\cA$ be a $\C$-algebra equipped with an involution, $\cD\subseteq \End(\cA)$ a $\C$-subalgebra of $\C$-linear endomorphisms of $\cA$ which contains the identity.
  \begin{enumerate}
      \item    We say that $\delta\in \cD$ is a $\cD$-differential operator of order $0$ if $\delta(ab)=\delta(a)b$ for all $a,b\in \cA$.
      %\item     We say that $\delta$ is a right multiplier if $\delta(ab)=a\delta(b)$ for all $a,b\in \cA$.
      %\item     The map $\delta$ is called a differential operator of order $0$ if there exists a finite family of linear maps 
      %$$\delta_{11},\cdots,\delta_{1r},\delta_{21},\cdots,\delta_{2r}:\cA\to \cA$$ such that $\delta_{11},\cdots,\delta_{1r}$ are left multipliers, $\delta_{21},\cdots,\delta_{2r}$ are right multipliers
      %and \begin{equation}\label{eqn:diff_0}\begin{aligned}
      %      \delta(ab)=\sum_{i=1}^r\delta_{1i}(a)\delta_{2i}(b).
      %  \end{aligned}\end{equation}
      \item We define $\cD$-differential operators of higher order by recurrence. A map $\delta\in \cD$ is called a $\cD$-differential operator of order $n\in \N$, 
      if there exists a finite family of linear maps
      $$\delta_{10},\cdots,\delta_{1r},\delta_{20},\cdots,\delta_{2r}\in \cD$$ which are $\cD$-differential operators of order $<n$ such that
    \begin{equation}\label{eqn:gen_dfndifferential operator}\begin{aligned}
            \delta(ab)=\delta(a)\delta_{20}(b)+\delta_{10}(a) \delta(b)+\sum_{i=1}^r\delta_{1i}(a)\delta_{2i}(b).
        \end{aligned}\end{equation}
        Furthermore, we suppose that 
            \begin{equation}\label{eqn:purestar}\begin{aligned}
              \delta_{20}(x^*)=\delta_{20}(x)^*,\quad\forall x\in \cA
            \end{aligned}\end{equation}
      %If $\delta_{20}=\mathrm{Id}$, then we say that $\delta$ is pure. %It will be convenient in our proofs to call the maps $\delta_{10}$ and $\delta_{20}$, the first and second twisting factors of $\delta$.
  \end{enumerate}
  We denote by $\Der(\cA)\subseteq \cD$ the space of $\cD$-differential operators of any order. 
\end{dfn}
\begin{rems}\label{rem:}
    Let us explain Definition \ref{dfn:differential operators} in more details. \begin{enumerate}
    \item We use a subspace $\cD$ of linear maps instead of the space of all linear endomorphisms of $\cA$, because in applications, it seems unhelpful to consider all linear endomorphisms. 
    Usually, one has a given subspace of linear endomorphisms of interest. 
    
   \item Condition \eqref{eqn:purestar} is included because our main theorem fails without it. 
    It plays an absolutely necessary role in the proof of Lemma \ref{lem:poly_cros}. In all our applications (Theorem \ref{thm:intro:Schwartz_closed_holo_Connes} and Corollary \ref{cor:Schwartz_func_nilpo_groups}), 
    the $\cD$-differential operators we consider have the property $\delta_{20}=c\mathrm{Id}$ with $c\in \{0,1\}$ and thus trivially satisfy \eqref{eqn:purestar}. 
    We include the general case of arbitrary $\delta_{20}$ satisfying \eqref{eqn:purestar} because the arguments in this section don't change with its inclusion.

    \item It is worth nothing that if $\delta$ is a $\cD$-differential operator of order $n$, then it is also a $\cD$-differential operator of order $m$ for any $m>n$.
    It might be tempting to define the order of a $\cD$-differential operator $\delta$ to be the minimal $n$ such that $\delta$ is of order $n$. We won't do so because the order behaves in a rather unintuitive way. 
    For example, the sum of a $\cD$-differential operator of order $n$ and another of order $m$ isn't necessarily of order $\max(n,m)$. It is always of order $\max(n,m)+1$. See the proof of Proposition \ref{prop:higher_deriv}. 
  \end{enumerate} 
\end{rems}

  %\begin{ex}\label{ex:examples}
  % If $\delta:\cA\to \cA$ satisfies $\delta(ab)=\delta(a)b+ca\delta(b)$, where $c\in \cA$, then $\delta$ is a differential operator of order $1$ with $r=0$, $\delta_{10}(x)=cx$, $\delta_{20}=\mathrm{Id}$.
  %\end{ex}
\begin{rem}\label{rem:diff}
  The following observation will be used in the proof of Proposition \ref{prop:higher_deriv}.
  If $\delta\in \cD$ such that there exists $\cD$-differential operators $\delta_{10},\cdots,\delta_{1r},\delta_{20},\cdots,\delta_{2r}\in \cD$ such that \eqref{eqn:gen_dfndifferential operator} and \eqref{eqn:purestar} hold,
   then $\delta$ is a $\cD$-differential operator 
  of order one plus the maximum of the orders of $\delta_{10},\cdots,\delta_{1r},\delta_{20},\cdots,\delta_{2r}$. Whenever we apply this remark, to make it clearer how it is applied, we colored the 
  term $\delta(a)\delta_{20}(b)$ by \textcolor{red}{red},  the term $\delta_{10}(a)\delta(b)$ by \textcolor{green}{green},  and the rest by \textcolor{blue}{blue}.
\end{rem}

\begin{prop}\label{prop:higher_deriv}The space $\Der(\cA)$ is a unital subalgebra of $\cD$ and thus of $\End(\cA)$ as well.
\end{prop}
\begin{proof}
It is clear that $\id$ is a $\cD$-differential operator of order $0$.
  Let $\delta$ and $\delta'$ be $\cD$-differential operators of order $n$ and $m$ respectively with
 $\delta_{10},\cdots,\delta_{1r},\delta_{20},\cdots,\delta_{2r}$ and $\delta_{10}',\cdots,\delta_{1r}',\delta_{20}',\cdots,\delta_{2r}'$ as in Definition \ref{dfn:differential operators}. Then,
    \begin{equation*}\begin{aligned}
        \delta(ab)+\delta'(ab)&=\color{blue}{\delta(a)\delta_{20}(b)}\color{black}+\color{blue}\delta_{10}(a) \delta(b)\color{black}+\color{blue}\sum_{i=1}^r\delta_{1i}(a)\delta_{2i}(b)\\
        &\color{black}+\color{blue}\delta'(a)\delta_{20}'(b)\color{black}+\color{blue}\delta_{10}'(a) \delta'(b)\color{black}+\color{blue}\sum_{i=1}^{r'}\delta'_{1i}(a)\delta'_{2i}(b).
    \end{aligned}\end{equation*}
    Hence, $\delta+\delta'$ is a $\cD$-differential operator of order $\max(n,m)+1$ by Remark \ref{rem:diff}. Notice here that there are only blue terms. This is possible because, one already supposes that $\delta$ and $\delta'$ are $\cD$-differential operators.

    We prove that $\delta\circ\delta'$ is a $\cD$-differential operator by induction on the order. 
    The case $n=m=0$ is clear. We prove the case $m=0$ by induction on $n$. We have 
        \begin{equation}\begin{aligned}
            \delta(\delta'(ab))=\delta(\delta'(a)b)=\color{red}\delta(\delta'(a))\delta_{20}(b)\color{black}+\color{blue}\delta_{10}(\delta'(a)) \delta(b)\color{black}+\color{blue}\sum_{i=1}^r\delta_{1i}(\delta'(a))\delta_{2i}(b)
        \end{aligned}\end{equation}
        All the terms $\delta_{10}\circ \delta',\cdots,\delta_{1r}\circ \delta'$ are $\cD$-differential operators by recurrence. Hence, by Remark \ref{rem:diff}, $\delta\circ\delta'$ is also a $\cD$-differential operator.
      
        The case $n=0$ is also proved by induction on $m$. We have 
        \begin{equation}\begin{aligned}
            \delta(\delta'(ab))=\color{red}\delta(\delta'(a))\delta'_{20}(b)\color{black}+\color{blue}\delta(\delta'_{10}(a)) \delta'(b)\color{black}+\color{blue}\sum_{i=1}^{r'} \delta(\delta'_{1i}(a))\delta'_{2i}(b).
        \end{aligned}\end{equation}
        All the terms $\delta\circ \delta'_{10},\cdots, \delta\circ\delta'_{1r'}$ are $\cD$-differential operators by recurrence. Hence, by Remark \ref{rem:diff}, $\delta\circ\delta'$ is again a $\cD$-differential operator.

    We can now prove the general case of the composition of $\cD$-differential operators by induction on $(n,m)$. We suppose that it is true for any $n',m'\in \N\cup\{0\}$ such that $n'\leq n$ and $m'\leq m$ and $n'+m'<n+m$. One has 
    \begin{equation}\begin{aligned}
        \delta(\delta'(ab))&=\color{red}\delta(\delta'(a))\delta_{20}(\delta_{20}'(b))\color{black}+\color{blue}\delta_{10}(\delta'(a)) \delta(\delta_{20}'(b))\color{black}+\color{blue}\sum_{i=1}^r\delta_{1i}(\delta'(a))\delta_{2i}(\delta_{20}'(b)).\\
        &\color{black}+\color{blue}\delta(\delta_{10}'(a))\delta_{20}(\delta'(b))\color{black}+\color{green}\delta_{10}(\delta_{10}'(a)) \delta(\delta'(b))\color{black}+\color{blue}\sum_{i=1}^r\delta_{1i}(\delta_{10}'(a))\delta_{2i}(\delta'(b)).\\
        &\color{black}+\color{blue}\sum_{j=1}^{r'}\Big(\delta(\delta'_{1j}(a))\delta_{20}(\delta'_{2j}(b))\color{black}+\color{blue}\delta_{10}(\delta'_{1j}(a)) \delta(\delta'_{2j}(b))\color{black}+\color{blue}\sum_{i=1}^r\delta_{1i}(\delta'_{1j}(a))\delta_{2i}(\delta'_{2j}(b))\Big).
         \end{aligned}\end{equation}
  Again by Remark \ref{rem:diff}, $\delta\circ\delta'$ is a $\cD$-differential operator. 
  
Finally, since the map $a\mapsto \lambda a$ is a differential operator of order $0$ for any $\lambda\in \C$, it follows that $\lambda \delta$, which is a composition of two $\cD$-differential operators, is also a $\cD$-differential operator. 
This finishes the proof that $\Der(\cA)$ is a unital subalgebra of $\cD$.:
\end{proof}
\begin{dfn}\label{dfn:stable}
    We say that a unital subalgebra $\cD$ of $\End(\cA)$ is stable if $\cD=\Der(\cA)$.
\end{dfn}
It is clear from Definition \ref{dfn:differential operators} that for any unital subalgebra $\cD$ of $\End(\cA)$, $\Der(\cA)$ is stable.

%Now suppose $\cA$ has an involution. It is clear that $$\delta^*+\delta^{\prime *}=(\delta+\delta')^*,\quad (\delta^*)^*=\delta,\quad (\delta\circ\delta')^*=\delta^*\circ\delta^{\prime *}.$$ 
%Furthermore, if $\delta$ is a right or a left multiplier then $\delta^*$ is a left or a right multiplier respectively. If $\delta$ is a $\cD$-differential operator of order $0$, then 
%    \begin{equation}\begin{aligned}
%        \delta^*(ab)=\sum_{i=1}^r\delta_{2i}^*(a)\delta_{1i}^*(b)
%    \end{aligned}\end{equation}
%  Hence, $\delta^*$ is again a $\cD$-differential operator of order $0$. The general case follows by induction.
We now give natural examples of stable subalgebras. Let $\mathfrak{g}$ be a nilpotent Lie algebra. The Baker-Campbell-Hausdorff formula is a finite sum because $\mathfrak{g}$ is nilpotent. Hence, we can view $\mathfrak{g}$ as a nilpotent Lie group. 
Let $\cA=C^\infty_c(\mathfrak{g})$. We equip $\cA$ with the convolution product after fixing some Haar measure on $\mathfrak{g}$. 
We remark that a Haar measure is equivalently a Lebesgue measure because $\mathfrak{g}$ is nilpotent.
\begin{prop}\label{prop:poly_nil_grp} For $D$ a differential operator on $\mathfrak{g}$ with polynomial coefficients, let $$\delta_D:\cA\to \cA,\quad \delta_D(f)=D(f).$$
  The space $\cD$ of all maps $\delta_D$ as $D$ varies is a stable unital subalgebra of $\End(\cA)$, i.e., $\delta_D$ is a $\cD$-differential operator for every $D$.
\end{prop}
\begin{proof}
  It is clear that $\cD$ is a unital subalgebra. We now show that it is stable. For $X\in \mathfrak{g}$, let $X_R$ be the associated right invariant vector field. 
One easily sees that $\delta_{X_R}$ is a $\cD$-differential operator of order $0$. We can write $D$ as a sum of terms of the form $PX_{1R}\cdots X_{kR}$, where $P$ is a polynomial and $X_{1R},\cdots, X_{kR}$ are right-invariant vector fields.
 Hence, by Proposition \ref{prop:higher_deriv}, it suffices to prove that if $D(f)=Pf$, then $\delta_D$ is a $\cD$-differential operator.
Again, by Proposition \ref{prop:higher_deriv}, we can suppose that $P$ is linear.  Let $\mathfrak{g}^1=\mathfrak{g}$ and $\mathfrak{g}^{n+1}=[\mathfrak{g}^n,\mathfrak{g}]$ be the central series of $\mathfrak{g}$. 
Since $\mathfrak{g}$ is nilpotent, there exists $k\in \N$ such that $\mathfrak{g}^{k+1}=0$. 
We choose a linear decomposition $\mathfrak{g}=V^1\oplus\cdots \oplus V^k$ such that $\mathfrak{g}^n=V^n\oplus\cdots\oplus V^k$. It is enough to consider $P\in V^{i*}$ for some $i$. 
We proceed by induction on $i$. For $i=1$, $\delta_P$ is a $\cD$-differential operator of order $1$. In fact 
    \begin{equation}\begin{aligned}
        \delta_P(f\star g)=\delta_P(f)\star g+f\star \delta_P(g).
    \end{aligned}\end{equation}
 The induction case follows by using the Baker-Campbell-Hausdorff formula like in \eqref{eqn:Heisenberg}.
\end{proof}
%\begin{rem}It is clear from the proof of Proposition \ref{prop:poly_nil_grp} that the order of $\delta_D$ isn't at all connected to the order of $D$ but rather to the structure of $\mathfrak{g}$ itself.
%\end{rem}

%Te adjoint of a differential operator is not a differential operator in general as can be seen by looking at the adjoint of order $0$ differential operator.
\textbf{From now on, we suppose that $\cA$ is a dense $*$-subalgebra of a $C^*$-algebra $A$.}
\begin{dfn}\label{dfn:cDcA}Let $\cD$ be a stable unital subalgebra of $\End(\cA)$. 
  We denote by $\cD(\cA)$ the set of $a\in A$ such that there exists a sequence $(x_n)_{n\in \N}\in \cA$ such that $x_n\to a$ and for every $\delta\in \cD$, $\delta(x_n)$ and $\delta(x_n^*)$ converge in $A$.
\end{dfn}
By taking constant sequences, we see that $\cA\subseteq \cD(\cA)$.% Furthermore, $\cD(\cA)$ is a complete locally convex topological space where the topology is given by the seminorms $$
\begin{rem}\label{rem:closure_of_graph}
  Let $\delta\in \cD$. Since $\delta$ is generally an unbounded operator, it is desirable to take the closure of the graph of $\delta$. 
 Some care should be taken, as there exists derivations which aren't closable. See \cite{UnboundedDerivations1,UnboundedDerivations2}
 for examples and further discussion of this issue. This issue doesn't concern us as we only care about the domain of the closure and not the linear map on the closure.
 \end{rem}
In the next theorem, $M_n(\cA)$, $M_n(\cD(\cA))$ and $M_n(A)$ denotes the space of $n\times n$ matrices with coefficients in $\cA$, $\cD(\cA)$ and $A$ respectively.
\begin{theorem}\label{thm:funccalc_high_der}
  Let $\cD$ be a stable unital subalgebra of $\End(\cA)$. The set $\cD(\cA)$ is a dense $*$-subalgebra of $A$. It is also closed under smooth functional calculus, i.e,\begin{enumerate}
    \item If $n\in \N$, $a\in M_n(\cD(\cA))$ is normal ($a^*a=aa^*$) and $f$ is a smooth function defined on some neighborhood of $\mathrm{spec}(a)$ with $f(0)=0$, then $f(a)\in  M_n(\cD(\cA))$.
    \item If $n\in \N$, $a\in M_n(\cD(\cA))$ and $f$ is a holomorphic function defined on some neighborhood of $\mathrm{spec}(a)$ with $f(0)=0$, then $f(a)\in M_n(\cD(\cA))$.
\end{enumerate}
Here $\mathrm{spec}(a)$ is the spectrum of $a$ as an element of $M_n(A)$.
\end{theorem}
For clarity of the exposition, the proof of Theorem \ref{thm:funccalc_high_der} will be divided into several lemmas
\begin{lem}\label{lem:algebra}
    The space $\cD(\cA)$ is a dense $*$-subalgebra of $\cA$.
\end{lem}
\begin{proof}
  The space $\cD(\cA)$ is dense because $\cA\subseteq \cD(\cA)$. Proving that $\cD(\cA)$ is closed under addition, involution, multiplication by scalars is straightforward.
  Let $a,b\in \cD(\cA)$, $(x_n)_{n\in \N},(y_n)_{n\in \N}\in \cA$ such that $x_n\to a$ and $y_n\to b$ and $\delta(x_n)$, $\delta(x_n^*)$, $\delta(y_n)$, $\delta(y_n^*)$ converge for every $\delta\in \cD$. \
  For any $\delta\in \cD$, we can find $\delta_{10},\cdots,\delta_{1r}$, $\delta_{20},\cdots,\delta_{2r}\in\cD$ such that
      \begin{equation}\begin{aligned}
          \delta(x_ny_n)=\delta(x_n)\delta_{20}(y_n)+ \delta_{10}(x_n)\delta(y_n)+\sum_{i=1}^r\delta_{1i}(x_n)\delta_{2i}(y_n).
      \end{aligned}\end{equation}
 Since $\delta_{1i}(x_n)$ and $\delta_{2i}(y_n)$ converge for every $i\in \{0,\cdots ,r\}$, we obtain that $\delta(x_ny_n)$ converges. By a similar argument, we deduce that $\delta(y_n^*x_n^*)$ converges. Hence, $ab\in \cD(\cA)$. 
 Therefore, $\cD(\cA)$ is a dense $*$-algebra.
\end{proof}
Let $n\in \N$, 
    \begin{equation}\begin{aligned}
        M_n:\End(\cA)\to \End(M_n(\cA)),\quad M_n(\delta)((a)_{1\leq i,j\leq n})=(\delta(a_{ij}))_{1\leq i,j\leq n}.
    \end{aligned}\end{equation}
Clearly $M_n$ is an algebra homomorphism.
\begin{lem}\label{lem:nqsidf}For any $n\in \N$, $M_n(\cD)$ is a stable subalgebra of $\End(M_n(\cA))$. Furthermore, $M_n(\cD)(M_n(\cA))=M_n(\cD(\cA))$.
\end{lem} 
\begin{proof}
  A straightforward computation shows that if $\delta$ is a left multiplier, then $M_n(\delta)$ is also a left multiplier. Similar computation shows that if \eqref{eqn:gen_dfndifferential operator} holds,  then 
    \begin{equation}\begin{aligned}
        M_n(\delta)(ab)=M_n(\delta)(a)M_n(\delta_{20})(b)+M_n(\delta_{10})(a)M_n(\delta)(b)+\sum_{i=1}^rM_n(\delta_{1i})(a)M_n(\delta_{2i})(b).
    \end{aligned}\end{equation}
  It is also clear that if $\delta_{20}$ satisfies \eqref{eqn:purestar}, then $M_n(\delta_{20})$ satisfies \eqref{eqn:purestar}. The lemma follows.
  \end{proof}
By Lemma \ref{lem:nqsidf}, it follows that to prove Theorem \ref{thm:funccalc_high_der}, it suffices to deal with the case $n=1$.
\begin{lem}\label{lem:qsdfjqsdjflksdqfjlksd}
    Let $a\in \cD(\cA)$ be self-adjoint and $\xi \in \R$. Then, $ae^{i\xi a}\in \cD(\cA)$. 
    Furthermore, we can find a family $(x_{n,\xi})_{n\in \N,\xi \in \R}\in \cA$ such that  \begin{enumerate}
      \item For every compact subset $K$ of $\R$, the sequence $(x_{n,\xi})_{n\in \N}$ converges uniformly in $\xi\in K$ to  $ae^{i\xi a}$ as $n\to +\infty$. Moreover, 
            \begin{equation}\label{bound_uniform_xnxi}\begin{aligned}
                \sup\{\norm{x_{n,\xi}}:n\in \N,\xi\in \R\}<+\infty
            \end{aligned}\end{equation}
        
        \item For every compact subset $K$ of $\R$ and for every $\delta\in \cD$, the sequence $\delta(x_{n,\xi})$ converges uniformly in $\xi\in K$ to an element in $A$  as $n\to +\infty$.
         We denote the limit by $\delta(ae^{i\xi a})$.
        \item For every $\delta\in \cD$, the function
            \begin{equation}\begin{aligned}
                \R\to A,\quad \xi\mapsto \delta(ae^{i\xi a})
            \end{aligned}\end{equation}
         is continuous. Furthermore, if $\delta$ is of order $l$, then
            \begin{equation}\label{eqn:sqjdifojqosdjfoqjsdiof}\begin{aligned}
                \sup\left\{\frac{\norm{\delta(x_{n,\xi})}}{|\xi|^l+1}:n\in \N,\xi\in \R\right\}<+\infty
            \end{aligned}\end{equation}        
            \item For any $\xi\in \R$, $n\in \N$, $x_{n,\xi}^*=x_{n,-\xi}$.  
    \end{enumerate}
\end{lem}
We remark that the notation $\delta(a e^{i \xi a})$ is an abuse of notation. In general $ae^{i\xi a}$ doesn't belong to $\cA$, 
 therefore the value of $\delta(a e^{i \xi a})$ usually depends on the choice of a sequence $x_n$ converging to $a$ as in Definition \ref{dfn:cDcA}, see Remark \ref{rem:closure_of_graph}.

We will delay the proof of Lemma \ref{lem:qsdfjqsdjflksdqfjlksd} for the moment. Instead, we will finish the proof of Theorem \ref{thm:funccalc_high_der}.1.
\begin{proof}[Proof of Theorem \ref{thm:funccalc_high_der}.1]
 To prove Theorem \ref{thm:funccalc_high_der}.1, we first deal with the self-adjoint case, then with the general case of a normal element. 
\begin{enumerate}
     \item    Let $a\in \cD(\cA)$ be self-adjoint, $f$ a $C^{\infty}$ function on a neighbourhood of $\spec(a)\subseteq \R$ with $f(0)=0$. 
     Since we are only interested in $f(a)$, without loss of generality, we can suppose $f\in C^{\infty}_c(\R)$. Since $f(0)=0$, we can  write $f(x)=xg(x)$ where $g\in C^{\infty}_c(\R)$. 
     The Fourier transform $\hat{g}$ of $g$ is a Schwartz function. 
     By the Fourier inversion formula $$f(x)=\frac{1}{2\pi}\int_{-\infty}^\infty xe^{ix\xi}\hat{g}(\xi)d\xi.$$Hence, 
         \begin{equation}\label{f(a)integral}\begin{aligned}
             f(a)=\frac{1}{2\pi}\int_{-\infty}^\infty ae^{ia\xi}\hat{g}(\xi)d\xi
         \end{aligned}\end{equation}
      and the integral is absolutely convergent. 
     We approximate the integral $\frac{1}{2\pi}\int_{-\infty}^\infty ae^{ia\xi}\hat{g}(\xi)d\xi$ by the finite Riemann sums $$u_n:=\frac{1}{2\pi}\sum_{j=-n^2}^{n^2}\frac{1}{n}ae^{ia \frac{j}{n}}\hat{g}\left(\frac{j}{n}\right)$$It is clear that $u_n\to f(a)$ as $n\to +\infty$.
     Let $(x_{n,\xi})_{n\in \N,\xi \in \R}$ be as in Lemma \ref{lem:qsdfjqsdjflksdqfjlksd}. We define $$s_n=\frac{1}{2\pi}\sum_{j=-n^2}^{n^2}\frac{1}{n}x_{n,\frac{j}{n}}\hat{g}\left(\frac{j}{n}\right).$$
      It is clear that $s_n\in \cA$. Since $x_{n,\xi}$ converges to $ae^{i\xi a}$ uniformly in $\xi$ on compact sets, and $\hat{g}$ is Schwartz, and by \eqref{bound_uniform_xnxi}, we deduce that $\lim_{n\to +\infty}\norm{s_n-u_n}=0$. Hence, $s_n\to f(a)$.
      Let $\delta\in \cD$. By Lemma \ref{lem:qsdfjqsdjflksdqfjlksd}.3, the integral 
          \begin{equation}\begin{aligned}
           \frac{1}{2\pi}\int_{-\infty}^{\infty}\delta(ae^{i\xi a})\hat{g}(\xi)d\xi   
          \end{aligned}\end{equation}
       is absolutely convergent. By using Riemann sums like we did with Integral \eqref{f(a)integral}, we deduce that $$\lim_{n\to +\infty}\delta(s_n)=\frac{1}{2\pi}\int_{-\infty}^{\infty}\delta(ae^{i\xi a})\hat{g}(\xi)d\xi.$$ 
       Since $x_{n,\xi}^*=x_{n,-\xi}$, we can reuse our argument to deduce that $$\lim_{n\to +\infty}\delta(s_n^*)=\frac{1}{2\pi}\int_{-\infty}^{\infty}\delta(ae^{i\xi a})\overline{\hat{g}(-\xi)}d\xi.$$  Hence, $f(a)\in \cD(\cA)$.
 \item Let $a\in \cD(\cA)$ be a normal element, $f$ a $C^{\infty}$ function on a neighbourhood of $\spec(a)\subseteq \C$ with $f(0)=0$. Like before, we can suppose that $f\in C^{\infty}_c(\C)$. 
      We write $z=x+iy$ for the coordinates in $\C$. Consider the function $g(z)=f(x)$. Then, $g(a)=f_{|\R}(\frac{a+a^*}{2})$. 
      It follows from the self-adjoint case that $g(a)\in \cD(\cA)$. Hence, by replacing $f$ with $f-g$, we can suppose that $f=0$ on $\R$. By a similar argument, we can suppose that $f=0$ on both $\R$ and $i\R$. Hence, 
      there exists $g\in C^\infty_c(\C)$ such that $f(z)=xyg(z)$. Since $g$ is smooth compactly supported, it follows that $\hat{g}$ is a Schwartz function. 
      Let $\Re(a)=\frac{a+a^*}{2}$ and $\Im(a)=\frac{a-a^*}{2i}$. By the Fourier inversion formula 
          \begin{equation}\label{eqn:wxncbwxc}\begin{aligned}
              f(a)=\frac{1}{(2\pi)^2}\int_{-\infty}^\infty\int_{-\infty}^\infty \Re(a)e^{i\xi \Re(a)}\Im(a)e^{i\eta \Im(a)}\hat{g}(\xi+i\eta)d\xi d\eta. 
          \end{aligned}\end{equation}
      
      Theorem \ref{thm:funccalc_high_der}.1 follows from Lemma \ref{lem:qsdfjqsdjflksdqfjlksd} applied to $\Re(a)$ and $\Im(a)$ by writing Integral \eqref{eqn:wxncbwxc} as a Riemann sum like we did in the self adjoint case.
This finishes the proof of Theorem \ref{thm:funccalc_high_der}.1 under the assumption of Lemma \ref{lem:qsdfjqsdjflksdqfjlksd}.\qedhere
  \end{enumerate}
\end{proof}

 The proof of Lemma \ref{lem:qsdfjqsdjflksdqfjlksd} requires a few lemmas. Let $a\in \cD(\cA)$ be a self-adjoint element. We fix a sequence $(x_n)_{n\in \N}$ such that $x_n\to a$ and $\delta(x_n)$ and $\delta(x_n^*)$ converge for every $\delta\in \cD$.
%By abuse of notation, for $m\in \N$, we define $\delta(a^m):=\lim_{n}\delta(x_n^m)$.  
By replacing $x_n$ with $\frac{x_n+x_n^*}{2}$, we can suppose that $x_n=x_n^*$. %Since $a\neq 0$, we can without loss of generality suppose that $x_n\neq 0$ for all $n\in \N$. Furthermore, since $x_n\to a$, it follows that $\frac{1}{\norm{x_n}}$ is bounded.
\begin{lem}\label{lem:hqusihdlf}For every $l\in \N$ and $\delta\in \cD$ of order $l$, there exists $C_1,C_2>0$ such that \begin{equation}\label{eqn:qsiopdfjqiojsdof}
  \norm{\delta(x_n^m)}\leq C_1m^lC_2^{m},\quad  \forall m,n\in \N.
 \end{equation}
 %One can take $C_2=\sup\{\norm{x_n}:n\in \N\}$.%\quad \norm{\delta(a^m)}\leq Cm^l\norm{a}^{m-1},
 \end{lem}
 \begin{proof}
We prove the lemma by induction on $l$. For $l=0$ and $m\in \N$, we have 
    \begin{equation}\label{eqn:jkmosUPFOIIOQSPD}\begin{aligned}
      \delta(x_n^{m+1})=\delta(x_n)x_n^{m}      
    \end{aligned}\end{equation}
Since $\norm{\delta(x_n)}$ and $\norm{x_n}$ converge, it follows that they are bounded. The case $l=0$ is finished.

Suppose the lemma holds for $\delta\in \cD$ of order $<l$. Let $ \delta\in \cD$ be of order $l$.
For each $i\in \{0,\cdots,r\}$, the Lemma holds for $\delta_{1i}$. Therefore, there exists constants $C_1,C_2$ such that

    \begin{equation}\label{eqn:qsdfjkojos}\begin{aligned}
        \norm{\delta_{1i}(x_n^m)}\leq C_1m^{\text{order of }\delta_{1i}}C_2^m\leq C_1m^{l-1}C_2^{m},\quad \forall n,m\in \N,i\in \{0,\cdots,r\}. 
    \end{aligned}\end{equation}
Let $C_1'$ be a constant such that
    \begin{equation}\begin{aligned}
      \norm{\delta(x_n)}\leq C_1' ,\quad \text{and}\quad  \norm{\delta_{2i}(x_n)}\leq C_1',\quad \text{and}\quad  C_1\leq C_1',\quad \forall n\in\N ,i\in \{1,\cdots ,r\}.
    \end{aligned}\end{equation}
Let $C_2'$ be a constant such that
    \begin{equation}\begin{aligned}
      \norm{\delta_{20}(x_n)}\leq C_2',\quad \text{and}\quad C_2\leq C_2',\quad \forall n\in\N.
    \end{aligned}\end{equation}
Some constants $C_1',C_2'$ exist because $\delta(x_n)$, $\delta_{20}(x_n),\cdots,\delta_{2r}(x_n)$ converge in $A$. 

By \eqref{eqn:gen_dfndifferential operator}, for $m\in \N$, we have
      \begin{equation}\label{eqn:A}\begin{aligned}
          \delta(x_n^{m+1})=\delta(x_n^m)\delta_{20}(x_n)+\delta_{10}(x_n^m)\delta(x_n)+\sum_{k=1}^r\delta_{1k}(x_n^{m})\delta_{2k}(x_n),\quad \forall n,m\in\N.
      \end{aligned}\end{equation}
      By recurrence, we deduce that
      \begin{equation}\label{eqn:inductiongder}\begin{aligned}
          \delta(x^{m+1}_n)=\delta(x_n)\delta_{20}(x_n)^{m}+\sum_{j=1}^m\delta_{10}(x^j_n)\delta(x_n)\delta_{20}(x_n)^{m-j}+\sum_{j=1}^m\sum_{k=1}^r\delta_{1k}(x^j_n)\delta_{2k}(x_n)\delta_{20}(x_n)^{m-j}
      \end{aligned}\end{equation}
Hence 
    \begin{equation}\begin{aligned}
        \norm{\delta(x_n^{m+1})}&\leq C_1'C_2^{\prime m}+\sum_{j=1}^m j^{l-1} C_1^{\prime 2}C_2^{\prime m}+\sum_{j=1}^m j^{l-1}rC_1^{\prime 2}C_2^{\prime m}\\
        &\leq C_1'C_2^{\prime m}+m^{l+1} C_1^{\prime 2}C_2^{\prime m}+m^{l+1}rC_1^{\prime 2}C_2^{\prime m}
    \end{aligned}\end{equation}
The result follows.
 \end{proof}

Let $\xi \in \R$, $\delta\in \cD$. By Lemma \ref{lem:hqusihdlf}, the sum 
    \begin{equation}\label{qksdfjlkqsjdkfljqmsd}\begin{aligned}
        \sum_{m=0}^{\infty}(i\xi )^m\frac{\delta(x_n^{m+1})}{m!}
    \end{aligned}\end{equation}
 converges in norm uniformly in $n$. By an abuse of notation, we denote the sum by $\delta(x_n e^{i\xi x_n})$.
 Since $\lim_{n\to +\infty}\delta(x_n^{m+1})$ exists for all $m\in \N$ and the convergence in \eqref{qksdfjlkqsjdkfljqmsd} is uniform, it follows that 
     \begin{equation}\label{eqn:sqdjfiojqsmdoifiosdjfqiojsd}\begin{aligned}
         \delta(ae^{i\xi a}):=\lim_{n\to +\infty} \delta(x_ne^{i\xi x_n})=\sum_{m=0}^{\infty}(i\xi )^m\frac{\lim_{n\to+\infty }\delta(x_n^{m+1})}{m!}
     \end{aligned}\end{equation}
     exists. %It also follows from Lemma \ref{lem:hqusihdlf}, that $\norm{\delta(e^{i\xi a})}\leq C'|\xi|^n$ for some constant $C'$ which only depends on $a$ and $\delta$.
\begin{lem}\label{lem:poly_cros} For every $l\in\N$, $\delta\in \cD$ of order $l$, there exists constant $C>0$ such that 
  $$\norm{\delta(x_ne^{i\xi x_n})}\leq C(|\xi|^l+1),\quad \forall n\in \N,\xi\in \R.$$
\end{lem}
\begin{proof}
We prove the lemma by induction on $l$.  If $l=0$, then by \eqref{eqn:jkmosUPFOIIOQSPD}
    \begin{equation}\begin{aligned}
      \delta(x_ne^{i\xi x_n})=\sum_{m=0}^{\infty}(i\xi )^m\frac{\delta(x_n^{m+1})}{m!}= \delta(x_n)\sum_{m=0}^{\infty}(i\xi )^m\frac{x_n^{m}}{m!}=\delta(x_n)e^{i\xi x_n}
    \end{aligned}\end{equation}
    Since $x_n$ are self-adjoint, it follows that $\norm{e^{i\xi x_n}}=1$. The case $l=0$ follows.

 Let $\delta\in \cD$ of order $l$. We suppose the lemma holds for $\cD$-differential operators of order $<l$.  By \eqref{eqn:inductiongder}, and the identity \begin{align}\label{eqn:jisdhipofsdfjiqpijf}
  \frac{1}{(j-1)!(m-j)!}\int_0^1s^{j-1}(1-s)^{m-j}ds=\frac{1}{m!},\quad 1\leq j\leq m,
  \end{align}
  it follows that 
\begin{align*}
 \delta(x_ne^{i\xi x_n})=\delta(x_n)e^{i\xi \delta_{20}(x_n)}&+i\xi \int_0^1\delta_{10}(x_ne^{is\xi x_n})\delta(x_n)e^{i(1-s)\xi \delta_{20}(x_n)}ds\\&+ i\xi\sum_{k=1}^r \int_0^1\delta_{1k}(x_ne^{is\xi x_n})\delta_{2k}(x_n)e^{i(1-s)\xi \delta_{20}(x_n)}ds
\end{align*}
Now we use crucially \eqref{eqn:purestar}, to deduce that $\delta_{20}(x_n)$ is self-adjoint. Therefore, $$\norm{e^{i\xi \delta_{20}(x_n)}}=\norm{e^{i(1-s)\xi \delta_{20}(x_n)}}=1.$$  The lemma follows from the induction hypothesis on $\delta_{1k}(x_ne^{is\xi x_n})$ for $k\in \{0,\cdots,r\}$.
\end{proof} 

\begin{proof}[Proof of Lemma \ref{lem:qsdfjqsdjflksdqfjlksd}]
  Let $\phi:\R_+\to \N$ be any function such which satisfies 
      \begin{equation}\begin{aligned}
          \phi(\xi)<m\implies \frac{\xi^m}{m!}\leq \frac{1}{\sqrt{m!}},\quad \forall\xi\in \R_+,m\in\N.
      \end{aligned}\end{equation}
  We take 
      \begin{equation}\begin{aligned}
          x_{n,\xi}:=\sum_{m=0}^{\phi(|\xi|)+n}(i\xi)^m\frac{x_n^{m+1}}{m!}.
      \end{aligned}\end{equation}
   The element $x_{n,\xi}$ is a finite sum of elements of $\cA$. Hence, it belongs to $\cA$. We now check that $x_{n,\xi}$ satisfies the required properties. 
  \begin{enumerate}
      \item  It is rather straightforward to show that $x_{n,\xi}$ converges to $ae^{i\xi a}$ uniformly on each compact subset of $\R$. Let us show \eqref{bound_uniform_xnxi}. Let $C=\sup\{\norm{x_n}:n\in \N\}$. It is finite because $x_n\to a$. Since $$\sum_{m=0}^{\infty}(i\xi)^m\frac{x_n^{m+1}}{m!}=x_ne^{i\xi x_n},$$ and $x_n$ is self-adjoint, we deduce that 
          \begin{equation}\begin{aligned}
              \norm{\sum_{m=0}^{\infty}(i\xi)^m\frac{x_n^{m+1}}{m!}}\leq \norm{x_n}\leq C.
          \end{aligned}\end{equation}
          By 
              \begin{equation}\label{eqn:B}\begin{aligned}
                  \norm{x_{n,\xi}-\sum_{m=0}^{\infty}(i\xi)^m\frac{x_n^{m+1}}{m!}}&=\norm{\sum_{m=\phi(|\xi|)+n+1}^{\infty}(i\xi)^m\frac{x_n^{m+1}}{m!}}\\
                  &\leq \sum_{m=\phi(|\xi|)+n+1}^{\infty}|\xi|^m\frac{C^{m+1}}{m!}\\&\leq \sum_{m=\phi(|\xi|)+n+1}^{\infty}\frac{C^{m+1}}{\sqrt{m!}}\\&\leq \sum_{m=0}^{\infty}\frac{C^{m+1}}{\sqrt{m!}}<+\infty
              \end{aligned}\end{equation}
            we deduce \eqref{bound_uniform_xnxi}.
      \item By Lemma \ref{lem:hqusihdlf}, it is again straightforward to show that $\delta(x_{n,\xi})$ converges to $\delta(ae^{i\xi a})$ uniformly on each compact subset of $\R$.
      \item The function $\xi\mapsto \delta(ae^{i\xi a})$ is a continuous function because it is the uniform limit (over compact subsets of $\R$) of continuous function by \eqref{eqn:sqdjfiojqsmdoifiosdjfqiojsd}. 
      Let us show \eqref{eqn:sqjdifojqosdjfoqjsdiof}. Let $\delta\in \cD$ of order $l$. By Lemma \ref{lem:hqusihdlf}, there exists $C_1,C_2$ such that $\norm{\delta(x_n^m)}\leq C_1m^lC_2^m$. Therefore, we have 
      \begin{equation}\label{eqn:C}\begin{aligned}
        \norm{\delta(x_{n,\xi})-\delta(x_ne^{i\xi x_n})}&=\norm{\sum_{m=\phi(|\xi|)+n+1}^{\infty}(i\xi)^m\frac{\delta(x_n^{m+1})}{m!}}\\
        &\leq \sum_{m=\phi(|\xi|)+n+1}^{\infty}|\xi|^m\frac{C_1(m+1)^lC_2^{m+1}}{m!}\\&\leq \sum_{m=\phi(|\xi|)+n+1}^{\infty}\frac{C_1(m+1)^lC_2^{m+1}}{\sqrt{m!}}\\&\leq \sum_{m=0}^{\infty}\frac{C_1(m+1)^lC_2^{m+1}}{\sqrt{m!}}<+\infty 
    \end{aligned}\end{equation}
    By Lemma \ref{lem:poly_cros}, we deduce \eqref{bound_uniform_xnxi}.   
      \item The identity $x_{n,\xi}^*=x_{n,-\xi}$ is trivial to verify.
  \end{enumerate}
  This finishes the proof of Lemma \ref{lem:qsdfjqsdjflksdqfjlksd}.
    \end{proof}
    We now proceed to prove Theorem \ref{thm:funccalc_high_der}.2. The proof of Theorem \ref{thm:funccalc_high_der}.2 relies on the fact that Theorem \ref{thm:funccalc_high_der}.1 is true even in a parametrized setting. 
    By this, we mean the following
  \begin{lem}\label{lem:ABCDEF}
  Let $K$ be a compact topological space, $(a_{w})_{w\in K}\subseteq \cD(\cA)$, $(x_{n,w})_{n\in \N,w\in K}\subseteq \cA$ a family of elements which satisfies the following:
  \begin{enumerate}
      \item For every $n\in \N,w\in K$, $x_{n,w}^*=x_{n,w}$.
      \item  For every $n\in \N$, the function $n\mapsto x_{n,w}$ is continuous in $w\in K$.
      \item The limit of $x_{n,w}$ as $n\to +\infty$ is equal to $a_{n,w}$ uniformly in $w\in K$.
      \item For every $\delta\in \cD$, the limit $\delta(x_{n,w})$ as $n\to +\infty$ exists in $A$ uniformly in $w\in K$
      \item For every $\delta\in \cD$, the limit $\delta(x_{n,w}^*)$ as $n\to +\infty$ exists in $A$ uniformly in $w\in K$          
    \end{enumerate}
    Let $f$ be a continuous complex-valued function defined in an open neighbourhood of $$\{(w,\lambda)\in K\times \C:\lambda\in \spec(a_w)\}$$ which is smooth in the $\lambda$ variable, and which satisfies $f(w,0)=0$.
    Then, there exists a family $(y_{n,w})_{n\in \N,w\in K}\subseteq \cA$ which satisfies the following: 
    \begin{enumerate}
    \item For every $n\in \N$, the function $n\mapsto y_{n,w}$ is continuous in $w\in K$.
    \item The limit of $y_{n,w}$ as $n\to +\infty$ is equal to $f(w,a_{w})$ uniformly in $w\in K$.
    \item For every $\delta\in \cD$, the limit $\delta(y_{n,w})$ as $n\to +\infty$ exists in $A$ uniformly in $w\in K$
\item For every $\delta\in \cD$, the limit $\delta(y_{n,w}^*)$ as $n\to +\infty$ exists in $A$ uniformly in $w\in K$          
  \end{enumerate}
  \end{lem}
    To prove Lemma \ref{lem:ABCDEF}, one follows the proof of Theorem \ref{thm:funccalc_high_der}.1 carefully making sure that adding a parameter doesn't cause any problems. 
    We leave this verification to the reader.

    The following is a counterpart to Lemma \ref{lem:qsdfjqsdjflksdqfjlksd} for non self-adjoint elements.
    \begin{lem}\label{lem:qsimodfjmioqjsmoidfjqko} Let $a\in \cD(\cA)$. For every $w\notin \mathrm{spec(a)}\cup \{0\}$, $\frac{a}{w-a}\in \cD(\cA)$. Furthermore, if $K\subseteq \C\backslash ( \mathrm{spec(a)}\cup\{0\})$ is a compact subset, then we can find a family $(x_{n,w})_{n\in \N,w\in K}\in \cA$ such that 
  the following is satisfied: \begin{itemize}
    %\item $w\mapsto x_n(w)$ is uniformly continuous on $K$
    \item The sequence $(x_{n,w})_{n\in\N}$ converges uniformly in $w\in K$ to $\frac{a}{w-a}$.
    %\item For every $n\in \N$, the map $w\mapsto x_{n,w}$ is continuous.
    \item For every $\delta\in \cD$, $\delta(x_{n,w})$ converges uniformly in $w\in K$ to an element denoted by $\delta(\frac{a}{w-a})$.
\item For every $\delta\in \cD$, $\delta(x_{n,w}^*)$ converges uniformly in $w\in K$ to an element denoted by $\delta(\frac{a^*}{\bar{w}-a^*})$.
    \item For every $\delta\in \cD$, the maps $w\mapsto \delta(\frac{a}{w-a})$ and $w\mapsto \delta(\frac{a^*}{\bar{w}-a^*})$ are continuous.
    \end{itemize}
    \end{lem}
    \begin{proof}Let $w\notin   \mathrm{spec(a)}\cup\{0\}$. The element $$a_{w}:=(w-a)^*(w-a)-|w|^2=-wa^*-\bar{w}a+a^*a\in \cD(\cA)$$ is self-adjoint whose spectrum is contained in $]-|w|^2,+\infty[$.
       By Theorem \ref{thm:funccalc_high_der}.1, applied to $f(x)=\frac{1}{x+|w|^2}-\frac{1}{|w|^2}$, we deduce that $f(a_{w})\in \cD(\cA)$.
       It follows that $$\frac{a}{w-a}=\left(f(a_w)+\frac{1}{|w|^2}\right)(w-a)^*a\in \cD(\cA).$$
      Lemma \ref{lem:qsimodfjmioqjsmoidfjqko} now follows from Lemma \ref{lem:ABCDEF}.
    \end{proof}
    \begin{proof}[Proof of Theorem \ref{thm:funccalc_high_der}.2]
      Let $a\in \cD(\cA)$, and $f$ a holomorphic function on a neighbourhood of $\spec(a)$ with $f(0)=0$. Since $f(0)=0$, $f(z)=zg(z)$ for some holomorphic function $g$. Let $\gamma$ be a contour around $\mathrm{spec}(a)$ such that 
      $\gamma$ is in the domain of $f$ and the Cauchy integral formula applies $$f(z)=\frac{z}{2\pi i}\int_\gamma \frac{g(w)}{w-z}dw.$$ We thus have $$f(a)=\frac{1}{2\pi i}\int_\gamma g(w)\frac{a}{w-a}dw.$$ 
      Without loss of generality, we can suppose that $0$ doesn't belong to the image of $\gamma$.
    
    Let $x_{n,w}$ be as in Lemma \ref{lem:qsimodfjmioqjsmoidfjqko} applied to $K$ equal to the image of $\gamma$. We approximate $f(a)$ by a Riemann sum, then approximate each term $\frac{a}{w-a}$ by $x_{n,w}$. This way we obtain a sequence $s_n\in \cA$ which converges to $f(a)$
     and such that $\delta(s_n)$ converges to $\frac{1}{2\pi i}\int_\gamma g(w)\delta(\frac{a}{w-a}) dw$ and $\delta(s_n^*)$ converges to $\frac{1}{2\pi i}\int_\gamma \bar{g}(w)\delta(\frac{a^*}{\bar{w}-a^*}) dw$.  The result follows.
    \end{proof}
%We say that $\delta$ is closed if $\delta$ is closed as a linear map between Banach spaces, i.e., if $x_n\in\dom(\delta) $ is a sequence such that $x_n\to x\in A$ and $\delta(x_n)\to y\in A$ then $x\in \dom(\delta)$ and $\delta(x)=y$.
\begin{cor}\label{cor:Schwartz_func_nilpo_groups}Let $\mathfrak{g}$, $A$ be as in Proposition \ref{prop:poly_nil_grp}, $\cS(\mathfrak{g})\subseteq A$ be the $*$-subalgebra of Schwartz functions. Then 
\begin{enumerate}
  \item If $a\in \cS(\mathfrak{g})$ is normal and $f$ is a smooth function defined on some neighbourhood of $\mathrm{spec}(a)$ with $f(0)=0$, then $f(a)\in \cS(\mathfrak{g})$.
\item If $a\in \cS(\mathfrak{g})$ and $f$ is a holomorphic function defined on some neighbourhood of $\mathrm{spec}(a)$ with $f(0)=0$, then $f(a)\in \cS(\mathfrak{g})$.
\end{enumerate}
\end{cor}
Corollary \ref{cor:Schwartz_func_nilpo_groups} is well known, see for example \cite{HulJenSchwartzOnNilpotenGrps}. A well known argument using holomorphic functional calculus (see for example \cite{BlackadarBookK}) implies the following corollary.
\begin{cor}The inclusion $i:\cD(\cA)\to A$ induces an isomorphism in $K$-theory $K_0(\cD(\cA))\simeq K_0(A)$.
\end{cor}
%\begin{rem}For $l\in \N$, one can also define $\cD_l(\cA)$ which consists of all $a\in A$ such that there exists $x_n\in \cA$ such that $x_n\to a$ and for every  $\delta\in \gDer_l(\cA)$, $\delta(x_n)$ and $\delta(x_n^*)$ converge. The proof of theorem \ref{thm:funccalc_high_der} proves as well the following 
%\begin{theorem}
%For every $l\in \N$, there exists $l'\in \N$ such that the set $\cD_l(\cA)$ is a $*$-subalgebra of $A$ which is closed under $C^{l'}$ functional calculus, i.e., \begin{enumerate}
%\item if $a\in \cD_l(\cA)$ and $f$ is a holomorphic function defined on some neighbourhood of $\mathrm{spec}(a)$ with $f(0)=0$, then $f(a)\in \cD_l(\cA)$.
%\item if $a\in \cD_l(\cA)$ is normal and $f$ is a $C^{l'}$ function defined on some neighbourhood of $\mathrm{spec}(a)$ with $f(0)=0$, then $f(a)\in \cD_l(\cA)$.
%\end{enumerate}
%\end{theorem}
%\end{rem}

  %Now suppose that $n\in \N$. If $\delta\in \gDer(\cA)$. We define $M_n(\delta):M_n(\cA)\to M_n(\cA)$ given by $\delta((a_{ij}))=(\delta(a_{ij}))$. It is straightforward to see that $M_n(\delta)\in  \gDer(M_n(\cA))$. Let $\tilde{\cD}$ be the set of all such generalized differential operators on $M_n(\cA)$. It is clear that $\tilde{\cD}$ is stable and $\tilde{\cD}(M_n(A))=M_n(\cD(A))$. This finishes the proof of Theorem \ref{thm:funccalc_high_der
  \section{Schwartz functions on Connes's tangent groupoid}\label{sec:SchwartzConnes}

  Let $M$ be a compact smooth manifold. The set $$\mathbb{T}M=M\times M\times \R_+^\times\sqcup TM\times \{0\}$$ can be equipped with the structure of a smooth manifold with boundary $TM\times \{0\}$ as follows. The subset 
  $M\times M\times \R_+^\times$ is declared an open subset of $\mathbb{T}M$ with its usual smooth structure. If $U\subseteq M$ is an open set, $\phi:U\to \R^{\dim(M)}$ is a diffeomorphism, then 
  $U\times U\times \R_+^\times\sqcup TU\times \{0\}$ is declared an open subset of $\mathbb{T}M$ and the following bijection is declared a diffeomorphism 
    \begin{equation}  \label{eqn:aux_chart}\begin{aligned}
      U\times U\times \R_+^\times\sqcup TU\times \{0\}&\to \R^{\dim(M)}\times \R^{\dim(M)}\times \R_+\\
      (y,x,t)&\mapsto \left(\frac{\phi(y)-\phi(x)}{t},\phi(x),t\right),\quad x,y\in U,t\in \R_+^\times\\
      (v,x,0)&\mapsto (d\phi_x(v),\phi(x),0),\quad x\in U,v\in T_xM
    \end{aligned}\end{equation}
  One can check that this defines a smooth structure on $\mathbb{T}M$. A more conceptual way to define the smooth structure is to notice that the smooth structure (and the topology) is uniquely determined by requiring that the maps \begin{align}\label{eqn:projTMCONNES}
   \mathbb{T}M\to M\times M\times\R_+,\quad (y,x,t)\mapsto (y,x,t),\quad (v,x,0)\mapsto (x,x,0)
  \end{align}
  and \begin{align}\label{eqn:dncf}
  \dnc(f):\mathbb{T}M\to \C,\quad (y,x,t)\mapsto \frac{f(y)-f(x)}{t},\quad (v,x,0)\mapsto df_x(v),
  \end{align}
  be smooth, where $f\in C^\infty(M)$.
  \paragraph{Differential operators.}
  Let $X\in \cX(M)$ be a vector field. We can define a vector field $\mathbb{X}$ on $\mathbb{T}M$ as follows. If $f\in C^\infty(\mathbb{T}M)$ is a smooth function, then we define $\mathbb{X}\cdot f\in C^\infty_c(\mathbb{T}M)$ by $$\mathbb{X}\cdot f (y,x,t)=tX_yf(y,x,t),\quad \mathbb{X}\cdot f(v,x,0)=X(x)\cdot f(v,x,0),$$ where $X_y$ means that $X$ acts on the $y$ variable, and  $X(x)\cdot f(v,x,0)$ means that we view $X(x)$ as a constant vector field on $T_xM$ which acts on the smooth function $f(\cdot,x,0)$. By the description of charts above, one checks that $\mathbb{X}$ is indeed a vector field.
  
  More generally, let $D$ be a differential operator of order $d$ acting on $M$. Then we can define a differential operator denoted $\mathbb{D}$ of order $d$ acting on $\mathbb{T}M$ by the formula 
  \begin{equation}\label{eqn:mathbbDConnes}
   \mathbb{D}\cdot f (y,x,t)=t^dD_yf(y,x,t),\quad \mathbb{D}\cdot f(v,x,0)=D(x)\cdot f(v,x,0),
  \end{equation} where $D_y$ means that $D$ acts on the $y$ variable, and  $D(x)\cdot f(v,x,0)$ means that we view the principal part of $D$ as a constant coefficient differential operator on $T_xM$ acting on the smooth function $f(\cdot,x,0)$. Locally by the principal part, we mean that if $D=\sum_{|I|\leq d}g_I\frac{\partial}{\partial x_I}$, then $D(x)=\sum_{|I|= d}g_I(x)\frac{\partial}{\partial x_I}$ seen as a constant coefficient differential operator on $T_xM$. Using the local charts of $\mathbb{T}M$, one can check that $\mathbb{D}$ is indeed a differential operator.
  %\paragraph{Groupoid structure}We equip the space $\mathbb{T}M$ with the structure of a Lie groupoid whose space of objects is $M\times \R_+$. We refer the reader to a the definition of Lie groupoids. The structure maps \begin{align*}
  %s:\mathbb{T}M\to M\times \R_+, &s(x,y,t)=(y,t),&s(x,X,t)=(x,0)\\
  %r:\mathbb{T}M\to M\times \R_+, &s(x,y,t)=(x,t),&s(x,X,t)=(x,0)\\
  %\mathrm{id}: M\times \R_+\to \mathbb{T}M, &\mathrm{id}(x,y,t)=(y,t),&s(x,X,t)=(x,0)
  %\end{align*}
  \paragraph{Convolution and adjoint.} From now on, we suppose that $M$ is equipped with a Riemannian metric. If $f,g\in \C^\infty_c(\mathbb{T}M)$, then we define their convolution $f\star g\in C^\infty_c(\mathbb{T}M)$ by the formula \begin{equation}\label{eqn:convlaw}
  \begin{aligned}
  f\star g(y,x,t)&=t^{-\dim(M)}\int_Mf(y,z,t)g(z,x,t)dz,\quad t\neq 0\\
  f\star g(v,x,0)&=\int_{T_xM}f(v-w,x,0)g(w,x,0)dw,
  \end{aligned}
  \end{equation}
  where the integral over $T_xM$ is with respect to the constant Riemannian metric on $T_xM$ induced from the Riemannian metric on $M$. We also define the adjoint by $$f^*(y,x,t)=\bar{f}(x,y,t),\quad f^*(v,x,0)=\bar{f}(-v,x,0)$$ 
  %Then $D$ can be written as sum of monomials of the form $X_1\cdots X_k$ with $X_i\in \cX(M)$ and $k\leq d$. One can then define a differential operator on $\mathbb{T}M$ where one replaces each such monomial by $t^{d-k}(tX_1)\cdots (tX_k)$. Notice that such an operator acts on $f(x,y,t)$ by $t^dD$ on the $x$ variable. This justifies the notation $t^dD$ for such an operator. The action of $D$ at $t=0$ is given by the principal symbol of $D$ seen
  We leave it to the reader to check that $C^\infty_c(\mathbb{T}M)$ with the operations defined above is a $*$-algebra.  
  We refer the reader to \cite{ConnesBook} and \cite{DebordLescureIndextheoreAndGroupoids} for more details on the convolution algebra of the tangent groupoid.
  
  \paragraph{Adjoint of Differential operators on $\mathbb{T}M$.}
  Let $D$ be a differential operator on $M$ of order $d$, $\mathbb{D}$ the associated differential operator on $\mathbb{T}M$. We	 sometimes denote $\mathbb{D}(f)$ by $\mathbb{D}\star f$. We also define $f\star \mathbb{D}$ by the formula $$f\star \mathbb{D}:=(\mathbb{D}^*\star f^*)^*,$$where $\mathbb{D}^*$ is the differential operator on $\mathbb{T}M$ associated to the formal adjoint $D^*$ of $D$. Equivalently $$f\star \mathbb{D}(y,x,t)=t^dD'_xf(y,x,t),\quad  f\star \mathbb{D}(v,x,0)=D(x)\cdot f(v,x,0),$$ where $D'$ is the formal transpose given by $$\int_M Df(x)g(x)dx=\int_Mf(x)D'g(x)dx,\quad  f,g\in C^\infty(M).$$
  One can check that \begin{equation}\label{eqn:multip_diff_connes}
   (\mathbb{D}\star f)\star g=\mathbb{D}\star(f\star g),\quad (f\star \mathbb{D})\star g=f\star (\mathbb{D}\star g)\quad f,g\in C^\infty_c(\mathbb{T}M).
  \end{equation}This justifies our notation $f\star \mathbb{D}$ and $\mathbb{D}\star f$.
  
  %We write explicitly for differential operators of order $0$, the formulas are quite simple. Let $f\in C^\infty(M)$. Then $$f\star g(x,y,t)=f(x)g(x,y,t),\quad f\star g(x,X,0)=f(x)g(x,X,0),\quad g\in C^\infty_c(\mathbb{T}M)$$ and $$ g\star f(x,y,t)=f(y)g(x,y,t),\quad g\star f(x,X,0)=f(x)g(x,X,0),\quad g\in C^\infty_c(\mathbb{T}M).$$ These operators are bounded by the uniform norm of $f$ and thus extends to a bounded operator $$f\star \cdot:C^*\mathbb{T}M\to C^*\mathbb{T}M,\quad \cdot\star f:C^*\mathbb{T}M\to C^*\mathbb{T}M.$$ 
  %The operator $\mathbb{D}$ also acts on the right on $C^\infty_c(\mathbb{T}M)$ by the formula $$f\mathbb{D}:=(\mathbb{D}^*f^*)^*,$$ where $\mathbb{D}^*$ is the operator on $\mathbb{T}M$ associated to the formal adjoint $D^*$ of $D$.
  
  \paragraph{Derivations in direction of $M$.} The functions $ \mathbb{D}\star f$ and $f\star \mathbb{D}$ agree at $t=0$. It follows that the map \begin{align*}
  \delta_{D}:C^\infty_c(\mathbb{T}M)\to C^\infty_c(\mathbb{T}M),\quad f\mapsto \frac{1}{t}(\mathbb{D}\star f-f\star \mathbb{D})
  \end{align*}
  is a well-defined differential operator on $\mathbb{T}M$ of order $d$. If $D$ is an order $0$ differential operator given by multiplication by $g\in C^\infty(M)$, then $\delta_g$ is an order $0$ differential operator given by multiplication by $\dnc(g)$, see \eqref{eqn:dncf}. For $X\in \cX(M)$, one can describe $\delta_X$ as follows. The flow of $X$ gives an $\R_+^\times$ action on $\mathbb{T}M$ by $$\beta_{\lambda}(y,x,t)=(\exp(\lambda X)\cdot y,\exp(\lambda X)\cdot x,t),\quad \beta_{\lambda}(v,x,0)=(d\exp(\lambda X)(x)(v),\exp(\lambda X)\cdot x,0).$$ The differential operator $\delta_{X}$ is the derivative of the action $\beta$ (plus multiplication by the divergence of $X$). Since $\delta_D$ is  given by a commutator, it satisfies the Leibniz rule \begin{equation}\label{eqn:LeibnizdeltaD}
  \delta_D(f\star g)=\delta_D(f)\star g+f\star \delta_D(g),\quad f,g\in C^\infty_c(\mathbb{T}M)
  \end{equation}
  and \begin{equation}\label{eqn:LeibnizdeltaD2}
  \delta_{D_1D_2}(f)=\delta_{D_1}(f)\star \mathbb{D}_2+\mathbb{D}_1\star \delta_{D_2}(f),\quad f\in C^\infty_c(\mathbb{T}M)
  \end{equation}
  
  %\paragraph{Multipliers of the $C^*$-algebra.}
  % It follows that from \eqref{eqn:multip_diff_connes}, that $\mathbb{D}$ defines an unbounded multiplier of $C^*\mathbb{T}M$. We will denote by $\overline{\mathbb{D}}$ the closure of $D$.  If $\mathbb{D}$ is of order $0$, the the multiplier $\mathbb{D}$ is bounded. %More precisely,
  \paragraph{Derivation in direction of $t$.} The group $\R_+^\times$ acts smoothly on the left on $\mathbb{T}M$ by the formula \begin{equation}\label{eqn:DebordSKandact}
   \alpha_\lambda:\mathbb{T}M\to \mathbb{T}M,\quad \alpha_\lambda(y,x,t)=(y,x,\lambda^{-1}t),\quad \alpha_\lambda(v,x,0)=(\lambda  v,x,0),\quad \lambda\in \R_+^\times.
  \end{equation} We also let $\R_+^\times$ acts on $C^\infty$ functions by the formula \begin{align*}
  \alpha_\lambda(f)=\lambda^{-\dim(M)} f\circ \alpha_{\lambda^{-1}},\quad f\in C^\infty_c(\mathbb{T}M),\lambda\in \R_+^\times.
  \end{align*}
  One has \begin{equation}\label{eqn:alpha_acion_law}
   \alpha_\lambda(f\star g)=\alpha_\lambda(f)\star \alpha_\lambda(g),\quad \alpha_\lambda(f)^*=\alpha_\lambda(f^*).
  \end{equation}
  % We define $\delta_{\alpha}:\dom(\delta)\subseteq C^*\mathbb{T}M\to C^*\mathbb{T}M$ to be the derivative of the $\R_+^\times$ action. So its domain is the set of $a\in C^*\mathbb{T}M$ such that $$\R_+^\times\to C^*\mathbb{T}M,\quad \lambda\mapsto\alpha_\lambda(a)$$ is $C^1$ and $\delta(a)$ is its derivative at $\lambda=1$. The map $\delta$ satisfies Leibniz rule $$\delta(a b)=\delta(a)b+a\delta(b),\quad a,b\in \dom(\delta)\subseteq C^*\mathbb{T}M.$$
  We define $\delta_{\alpha}$ to be the derivative of the $\R_+^\times$-action at $\lambda=1$. One can check that \begin{align*}
  \delta_{\alpha}(f)(y,x,t)&=-\dim(M)f(y,x,t)+t\frac{\partial}{\partial t}f(y,x,t)\\
  \delta_{\alpha}(f)(v,x,0)&=-\dim(M)f(v,x,0)-\sum_{i}v_i\frac{\partial}{\partial v_i}f(v,x,0),
  \end{align*}
  where $v_i$ are any local coordinates for $T_xM$. It follows from \eqref{eqn:alpha_acion_law} that $\delta_{\alpha}$ satisfies the Leibniz rule \begin{equation}\label{eqn:Leibnizeqn}
   \delta_{\alpha}(f\star  g)=\delta_{\alpha}(f)\star g+f\star\delta_{\alpha}(g),\quad f,g\in C^\infty_c(\mathbb{T}M).
  \end{equation}
   \begin{lem}\label{lem:f_iandX_i} One can find a finite family of smooth functions $f_1,\cdots,f_k\in C^\infty(M)$ and vector fields $X_1,\cdots,X_k\in \cX(M)$ such that for any $X\in \cX(M)$, $X=\sum_{i=1}^kX(f_i)X_i$.
     \end{lem}
  \begin{proof}
  Let $f_1,\cdots,f_k:M\to \R$ be smooth functions such that for all $x\in M$, $df_{1,x},\cdots,df_{k,x}$ span $T_x^*M$. Let $\theta=M\times \R^k$ be the trivial vector bundle on $M$ of rank $k$. The forms $df_i$ define a surjective bundle map $\theta\to T^*M$. Its dual is an injective bundle map $i:TM\to \theta$. Let $p:\theta\to TM$ be any bundle map such that $p\circ i=\mathrm{Id}$. The map $p$ determines the vector fields $X_1,\cdots,X_k$.
  %We cover $M$ with charts $U_1,\cdots,U_k$ which are all diffeomorphic to $\R^{\dim(M)}$. Let $p_i\in C^\infty_c(U_i)$ be smooth functions such that $\sum_{i=1}^kp_i^2=1$, where we extended each $p_i$ to $M$ by $0$ outside $U_i$. For each $i=1,\cdots,k$, and $j=1,\cdots,\dim(M)$, let $X_{ij}(x)=p_i(x)\frac{\partial}{\partial x_j}$  and $f_{ij}(x)=p_i(x)x_j$ where $x_j$ is the $j$th local coordinate of $x$ in $U_i$. One can check directly that $X_{ij}$ and $f_{ij}$ satisfy the lemma.
  \end{proof}
  We fix a choice of $f_i,X_i$ for the rest of this section. By Lemma \ref{lem:f_iandX_i}, if $x\in M$, then $\sum_{i=1}^kdf_i(x)X_i(x)$ is equal to the identity $T_xM\to T_xM$. By taking the trace of $\sum_{i=1}^kdf_i(x)X_i(x)$, we deduce that \begin{equation}\label{eqn:sumtrace}
  \sum_{i=1}^kX_i(f_i)(x)=\dim(M),\quad \forall x\in M
  \end{equation}
  For $f\in C^\infty_c(\mathbb{T}M)$, the functions $\delta_{\alpha}(f)$ and $-\dim(M)f-\sum_{i=1}^k\delta_{f_i}(\mathbb{X}_i\star f)$ agree at $t=0$. Hence, we can define \begin{equation}\label{eqn:dfn_hat_delta}\begin{aligned}
  \hat{\delta}(f):&=\frac{1}{t}\left(\delta_{\alpha}(f)+\dim(M)f+\sum_{i=1}^k\delta_{f_i}(\mathbb{X}_i\star f)\right)\\&=\frac{1}{t}\left(\delta_{\alpha}(f)+\sum_{i=1}^k\mathbb{X}_i\star \delta_{f_i}(f)\right),
  \end{aligned}
  \end{equation}
  where in the second equality we used \eqref{eqn:sumtrace}.
  \begin{lem}\label{lem:hatdelta_deriv} The map $\hat{\delta}$ satisfies $$\hat{\delta}(f\star  g)=\hat{\delta}(f)\star g+f\star\hat{\delta}(g)+\sum_{i=1}^k\delta_{X_i}(f)\star\delta_{f_i}(g),\quad f,g\in C^\infty_c(\mathbb{T}M)$$ and hence a derivation of order $2$ on $C^\infty_c(\mathbb{T}M)$.
  \end{lem}
  \begin{proof}
  By \eqref{eqn:LeibnizdeltaD} and \eqref{eqn:Leibnizeqn} \begin{equation*}
  \hat{\delta}(f\star  g)-\hat{\delta}(f)\star g-f\star\hat{\delta}(g)=\frac{1}{t}\left(\sum_{i=1}^k\mathbb{X}_i\star f\star \delta_{f_i}(g)-f\star \mathbb{X}_i\star \delta_{f_i}(g)\right)=\sum_{i=1}^k\delta_{X_i}(f)\star\delta_{f_i}(g)\qedhere
  \end{equation*}
  %Using \eqref{eqn:sumtrace}, one sees that $\sum_{i=1}^k\mathbb{X}_i\star \delta_{f_i}(g)-\delta_{f_i}(\mathbb{X}_i\star g)=\dim(M)g$. The result follows.
  \end{proof}
  The reader is invited to see the similarity between \eqref{eqn:dfn_hat_delta} and \cite[The formula for $T$ on Page 4]{HigsonHaj}.
  %\begin{rem} The vector field $\frac{\partial}{\partial t}$ on $M\times M\times \R_+^\times$ doesn't extend to a vector field on $\mathbb{T}M$, only $t\frac{\partial}{\partial t}$ does. We are thus lead to use $\delta$ instead in the definition of Schwartz functions in Section \ref{subsec:Connes_Schwartz_holo}.
  %\end{rem}
  %\paragraph{Derivative in the $t$ direction.}
  \paragraph{Schwartz functions.}
  
  \begin{dfn}\label{dfn:SchwartzConnes} 
    We define $\cS(\mathbb{T}M)$ to be the space of smooth function $f\in C^\infty(\mathbb{T}M)$ such that $f$ is bounded and all iterated applications of the following differential operators give bounded functions on $\mathbb{T}M$ \begin{itemize}
  \item the operator $f\mapsto tf$, where $t:\mathbb{T}M\to \R_+$ is the natural projection.
  \item the operator $f\mapsto \mathbb{D}\star f$, where $D$ is any differential operator on $M$
  \item the operator $f\mapsto \delta_{D}(f)$, where $D$ is any differential operator on $M$
  %\item the operator $f\mapsto \delta_{\alpha}(f)$.
  \item the operator $f\mapsto \hat{\delta}(f)$.
  \end{itemize}
  \end{dfn}
  We can add the application $f\mapsto \delta_{\alpha}(f)$ to the above list, but this is redundant as it follows from the others using \eqref{eqn:dfn_hat_delta}.
  \begin{prop}\label{prop:schwartz_at_point} Let $f\in \cS(\mathbb{T}M)$. Then for every $x\in M$, the function $v\mapsto f(v,x,0)$ is a Schwartz function on the vector space $T_xM$ in the classical sense.
  \end{prop}
  \begin{proof}
  If $X\in \cX(M)$ is a vector field, then $v\mapsto (\mathbb{X}\star f)(v,x,0)$ is the application of the constant vector field $X(x)$ to $v\mapsto f(v,x,0)$. While if $g\in C^\infty(M)$, then the map $v\mapsto \delta_g(f)(v,x,0)$ is the pointwise product of the map $v\mapsto f(v,x,0)$ with the linear map $dg_x:T_xM\to \C$. By iterated application of the previous two operations, it follows that $v\mapsto f(v,x,0)$ is Schwartz.
  \end{proof}
  \begin{ex}\label{ex:Schwartz_Connes} Let $g_1,\cdots,g_l:M\to \R$ be smooth functions such that the map $$M\mapsto \R^l,\quad x\mapsto (g_1(x),\cdots,g_l(x))$$ is an embedding. Then, the function $$e^{-\sum_{i=1}^l\dnc(g_i)^2-t^2}\in C^\infty(\mathbb{T}M)$$ is Schwartz, where $\dnc(g_i)$ is defined in \eqref{eqn:dncf}.
  \end{ex}
  The following proposition summarizes the main properties of Schwartz functions. In its proof, we remark that by \eqref{eqn:LeibnizdeltaD2}, in Definition \ref{dfn:SchwartzConnes} the second and third conditions can be replaced by $\mathbb{D}\star f$ and $ \delta_{D}(f)$ are bounded where $D$ is either a vector field or a smooth function on $M$.
  \begin{prop}\label{prop:mainprop_Schwartz} The following holds 
   \begin{enumerate}
   \item The definition of Schwartz functions doesn't depend on the choice of $f_i,X_i$ in Lemma \ref{lem:f_iandX_i}. 
  \item If $f\in \cS(\mathbb{T}M)$, then $f\in C_0(\mathbb{T}M)$, i.e. $f$ vanishes at infinity.
   \item If $f,g\in \cS(\mathbb{T}M)$, then the integral in $f\star g$ is absolutely convergent and $f\star g\in \cS(\mathbb{T}M)$
   \item If $f\in \cS(\mathbb{T}M)$, then $f^*\in \cS(\mathbb{T}M)$
   \end{enumerate}
  \end{prop}
  \begin{proof}
  For the first part, we first need some lemmas. \begin{lem}\label{lem:pointwiseprodSchwartz} If $f\in C^\infty_c(M\times M\times \R_+)$ and $g\in \cS(\mathbb{T}M)$, then the pointwise product $(f\circ \pi)g\in \cS(\mathbb{T}M)$, where $\pi:\mathbb{T}M\to M\times M\times \R_+$ is the map in \eqref{eqn:projTMCONNES}
  \end{lem}\begin{proof}
   $(f\circ \pi)g$ is bounded because $g$ and $f$ are bounded. We argue that each of the applications in Definition \ref{dfn:SchwartzConnes} gives functions of the form $(f'\circ \pi)g'$ for some $f'\in  C^\infty_c(M\times M\times \R_+)$ and $g'\in \cS(\mathbb{T}M)$. The result then follows by induction. We thus have \begin{itemize}
  \item the function $t(f\circ \pi)g=(f\circ \pi)(tg)$
  \item If $D$ is a function on $M$, then $\mathbb{D}\star ((f\circ \pi)g)=(f\circ \pi)(\mathbb{D}\star g)$. If $D$ is a vector field, then we have \begin{equation}\label{eqn:hqsdifjlqs}
   \mathbb{D}\star ((f\circ \pi)g)=(D_y(f)\circ \pi)(tg)+(f\circ \pi)(\mathbb{D}\star g),
  \end{equation} where $D_y(f)$ is the action of $D$ on the $y$ variable in $f(y,x,t)$. The equality can be checked directly on $M\times M\times \R_+^\times$ and by density equality follows on $\mathbb{T}M$.
  \item If $D$ is a function, then $\delta_{D}((f\circ \pi)g)=(f\circ \pi)\delta_D(g)$. If $D$ is a vector field, then one can as in \eqref{eqn:hqsdifjlqs} show that $$\delta_{D}((f\circ \pi)g)=(D_y(f)\circ \pi-D_x'(f)\circ \pi)(g)+(f\circ \pi)\delta_D(g)$$
  \item one has $$\delta_{\alpha}((f\circ \pi)g)=(f\circ \pi)\delta_{\alpha}(g)+(\frac{\partial}{\partial t}(f)\circ \pi)(tg)$$ and so $$\hat{\delta}((f\circ \pi)g)=(f\circ \pi)\hat{\delta}(g)+(\frac{\partial}{\partial t}(f)\circ \pi)(g)+\sum_{i=1}^k(X_{iy}(f)\circ \pi)(\delta_{f_i}(g))$$
  \end{itemize}
  This finishes the proof.
  \end{proof}
  \begin{lem}\label{lem:qsfjihqmosjdkofjio}Let $f\in C^\infty(M\times M)$ be a smooth function that vanishes to the order $1$ on the diagonal, and let $\tilde{f}\in C^\infty(\mathbb{T}M)$ be the smooth function given by $$(y,x,t)\mapsto \frac{f(y,x)}{t^2},\quad (v,x,0)\mapsto \frac{1}{2}d^2f_x(v),$$ where $d^2f$ is the Hessian of $f$. Then if $g\in \cS(\mathbb{T}M)$ then the pointwise product $\tilde{f}g\in \cS(\mathbb{T}M)$.
  \end{lem}
  \begin{proof}
  Any function $f$ which vanishes to the order $1$ can be written as a finite sum of functions of the form $h(y,x)(h_1(y)-h_1(x))(h_2(y)-h_2(x))$ where $h \in C^\infty(M\times M),h_1,h_2\in C^\infty(M)$. It follows that the pointwise product $\tilde{f}g$ is sum of pointwise product of $h\circ \pi$ with $\delta_{h_1}(\delta_{h_2}(g))$ where $\pi:\mathbb{T}M\to M\times M$ is the natural map. The result then follows from Lemma \ref{lem:pointwiseprodSchwartz}.
  \end{proof}
  \begin{lem} Let $g_1,\cdots,g_{k'}$ and $Y_1,\cdots,Y_{k'}$ be another family satisfying Lemma \ref{lem:f_iandX_i}. Then there exists a family of smooth functions $h_{ij}:M\to \R$ for $1\leq i\leq k$ and $1\leq j\leq k'$ such that \begin{equation}\label{eqn:qsdklnfmiojs}
   df_i=\sum_{j=1}^{k'}h_{ij}dg_{j},\quad Y_j=\sum_{i=1}^kh_{ij}X_i.
  \end{equation}

  \end{lem}

  \begin{proof}
  Let $x_1,\cdots,x_{\dim(M)}$ be local coordinates in $M$. Then $df_i=\sum_{l=1}^{\dim(M)}\phi_{il}dx_l$ and $X_i=\sum_{l=1}^{\dim(M)}\phi'_{il}\frac{\partial}{\partial x_l}$ for some functions $\phi_{il}$ and $\phi_{il}'$. 
  Similarly, $dg_j=\sum_{l=1}^{{\dim(M)}}\psi_{jl}dx_l$ and $Y_j=\sum_{l=1}^{{\dim(M)}}\psi'_{jl}\frac{\partial}{\partial x_l}$. Then one can take $h_{ij}=\sum_{l=1}^{{\dim(M)}}\phi_{il}\psi_{jl}'$. To define $h$ globally one takes a partition of unity.
  \end{proof}
  We can now prove Proposition \ref{prop:mainprop_Schwartz}.1.  Let $\hat{\delta}'$ be the operator associated to $g_1,\cdots,g_{k'}$ and $Y_1,\cdots,Y_{k'}$. We have  \begin{align*}
   \hat{\delta}(f)-\hat{\delta}'(f)&=\frac{1}{t}\left(\sum_{i=1}^k\delta_{f_i}(\mathbb{X}_i\star f)-\sum_{j=1}^{k'}\delta_{g_j}(\mathbb{Y}_j\star f) \right)\\&=\sum_{i=1}^k\frac{1}{t^2}\left(\left(f_i(y)-f_i(x)-\sum_{j=1}^{k'}(g_j(y)-g_j(x))h_{ij}(y)\right)(\mathbb{X}_i\star f) \right) .
  \end{align*}
  By \eqref{eqn:qsdklnfmiojs}, the function $f_i(y)-f_i(x)-\sum_{j=1}^{k'}(g_j(y)-g_j(x))h_{ij}(y)$ vanishes to order $1$ on the diagonal. By Lemma \ref{lem:qsfjihqmosjdkofjio}, $\hat{\delta}'(f)\in \cS(\mathbb{T}M)$. Proposition \ref{prop:mainprop_Schwartz}.1 follows.
  
  For Proposition \ref{prop:mainprop_Schwartz}.2, suppose that $f\notin  C_0(\mathbb{T}M)$. Then there exists  $\epsilon>0$ a sequence in $\mathbb{T}M$ and which goes to infinity yet $|f|\geq \epsilon$ for every element in the sequence. By passing to a subsequence we can suppose that the sequence is either of the form $(y_n,x_n,t_n)$ or of the form $(v_n,x_n,0)$. Suppose we have the first case. Then by taking a subsequence, we can suppose that $x_n\to x$ and $y_n\to y$ and $t_n\to t\in  [0,+\infty]$. If $t=+\infty$, we get a contradiction to the fact that $tf$ is bounded. If $t\in ]0,+\infty[$, then the sequence $(y_n,x_n,t_n)$ converges in $\mathbb{T}M$ to $(y,x,t)$, again a contradiction. If $x\neq y$ then we get a contradiction to the fact that $\delta_{g}(f)$ is bounded where $g\in C^\infty(M)$ is any smooth function with $g(x)=g(y)$. So we have $t=0$ and $x=y$. The sequence $(y_n,x_n,t_n)$ converging to infinity implies that there exists $g\in C^\infty(M)$ such that $\left|\frac{g(y_n)-g(x_n)}{t_n}\right|\to +\infty$. We get then a contradiction to the fact that $\delta_{g}(f)$ is bounded. The case of a sequence $(v_n,x_n,0)$ is similar.
  
  For Proposition \ref{prop:mainprop_Schwartz}.3, first the integral in $f\star g$ is absolutely convergent by Proposition \ref{prop:schwartz_at_point}. Hence, $f\star g$ is a well-defined function on $\mathbb{T}M$. We now show its continuity. Let $g_i$ be as in Example \ref{ex:Schwartz_Connes}. Then consider the function \begin{equation}\label{eqn:phi}
   \phi=(1+\dnc(g_1)^2+\cdots+\dnc(g_l)^2)^{\frac{\dim(M)}{2}+1}\in C^\infty(\mathbb{T}M).
  \end{equation}
  The function $f$ being Schwartz implies that $\phi f$ is bounded. One can easily show by looking at local coordinates of $\mathbb{T}M$ that there exists $C>0$ (only depends on $\phi$) such that $$\norm{f\star g}_{\infty}\leq C\norm{\phi f}_{\infty}\norm{g}_{\infty}.$$ Since $\phi f$ and $g$ are Schwartz functions, $\phi f,g\in C_0(\mathbb{T}M)$. We can thus approximate them uniformly with compactly supported functions. It follows that $f\star g\in C_0(\mathbb{T}M)$. Smoothness of $f\star g$ as well as the fact that $f\star g\in \cS(\mathbb{T}M)$ follow easily from \eqref{eqn:multip_diff_connes}, \eqref{eqn:LeibnizdeltaD}, \eqref{eqn:alpha_acion_law}  and Lemma \ref{lem:hatdelta_deriv}.
   
  For Proposition \ref{prop:mainprop_Schwartz}.4, since $f^*$ is bounded if $f$ is bounded, the result follows from the following identities \begin{align*}
  tf^*=(tf)^*,\quad \mathbb{D}\star f^*=(-t\delta_{D^*}(f)+\mathbb{D}^*\star f)^*,\quad \delta_{D}(f^*)=-\delta_{D^*}(f)^*
  ,\quad \delta_{\alpha}(f^*)=\delta_{\alpha}(f)^*, 
  \end{align*}
  and the identity \begin{align*}
  \hat{\delta}(f^*)^*-\hat{\delta}(f)&=\frac{1}{t}\left(\sum_{i=1}^k-\delta_{f_i}(f\star \mathbb{X}_i^*)-\delta_{f_i}(\mathbb{X}_i\star f)\right)\\&=\frac{1}{t}\left(\sum_{i=1}^k\delta_{f_i}(f\star \mathbb{X}_i)-\delta_{f_i}(\mathbb{X}_i\star f)\right)+\sum_{i=1}^k\delta_{f_i}(f\star \mathrm{div}(X_i))\\&=\sum_{i=1}^k-\delta_{f_i}(\delta_{X_i}(f))+\delta_{f_i}(f\star \mathrm{div}(X_i))\qedhere
  \end{align*}
  \end{proof}
  \begin{rem}\label{rem:diff_dfns_schwartz}
      There are different definitions in the literature of Schwartz functions. Our definition doesn't precisely agree with \cite{PaoloSchwartzAlgebraTangent}. In \cite{PaoloSchwartzAlgebraTangent}, the author adds a conical support condition which we don't need.
      Our definition agrees with the one proposed by Debord and Skandalis \cite{DebordSkandalis1}. We refer the reader to \cite[Section 1.6]{DebordSkandalis1} for a treatment of the Schwartz functions defined here using classical semi-norm estimates.

       We give here a definition of Schwartz functions which is equivalent to ours by looking in local coordinates. 
       A function $f\in \mathcal{S}(\mathbb{T}M)$ is Schwartz if and only if it satisfies the following:
       \begin{enumerate}
           \item   For every $ k,l\in \N$, $D$ a differential operator on $M\times M$, $$\sup\left\{\left|t^k\frac{d}{dt^l}Df(y,x,t)\right|:(y,x,t)\in M\times M\times [1,+\infty[\right\}<+\infty$$
           \item  For every $k, l\in \N$, $D$ a differential operator on $M\times M$, $K\subseteq M\times M\backslash M$ a compact subset outside the diagonal,
            $$\sup\left\{\left|t^{-k}\frac{d}{dt^l}Df(y,x,t)\right|:(y,x,t)\in K\times ]0,1]\right\}<+\infty$$
            \item For every $U\subseteq M$ open subset diffeomorphic to $\R^{\dim(M)}$ by a map $\phi:U\to \R^{\dim(M)}$, $\Phi$ the local chart of $\mathbb{T}M$ associated to $\phi$ defined in \eqref{eqn:aux_chart}, $K\subseteq \R^{\dim(M)}$ a compact subset, $k,l\in \N$, $\alpha,\beta\in \N^{\dim(M)}$, one has 
                  $$\sup\left\{\left|\norm{v}^k\frac{d}{dt^l}\frac{d}{dv^\alpha}\frac{d}{dx^\beta}f\circ \Phi^{-1}(v,x,t)\right|:(v,x,t)\in \hat{K}\right\}<+\infty,$$
                  where $\hat{K}=\{(v,x,t)\in \R^{2\dim(M)+1}:x,x+tv\in K\}$.
       \end{enumerate}
  \end{rem}

   \paragraph{$C^*$-algebra of the tangent groupoid.}
   Let $f\in C^\infty_c(\mathbb{T}M)$, $t\neq 0$. We define the operator \begin{align*}
   \pi_t(f):L^2M&\to L^2M\\
   g&\mapsto (y\mapsto t^{-\dim(M)}\int_M f(y,x,t)g(x)dx),\quad g\in L^2(M) 
  \end{align*}
  For each $x\in M$, we also define the operator \begin{align*}
   \pi_{x}(f):L^2T_xM&\to L^2T_xM\\
   g&\mapsto (v\mapsto \int_{T_xM} f(v-w,x,0)g(w)dw),\quad g\in L^2(T_xM) 
  \end{align*}
  We then define $$\norm{f}:=\max{\left(\sup_{t\in \R_+^\times}\norm{\pi_t(f)},\sup_{x\in M}\norm{\pi_{x}(f)}\right)}.$$ 
  We define the $C^*$-algebra $C^*\mathbb{T}M$ to be the completion of $C^\infty_c(\mathbb{T}M)$ with respect to $\norm{\cdot}$. The $C^*$-algebra $C^*\mathbb{T}M$ lies in a short exact sequence $$0\to \mathcal{K}(L^2M)\otimes C_0(\R_+^\times)\to C^*\mathbb{T}M\to C_0(T^*M)\to 0,$$ where $\cK(L^2M)$ is the $C^*$-algebra of compact operators on $L^2M$, see \cite[Proposition 5 Page 108]{ConnesBook}% The action of $\R_+^\times$ extends to an action by unitary automorphisms on $C^*\mathbb{T}M$.
  \begin{prop}The space $\cS(\mathbb{T}M)$ is a $*$-subalgebra of $C^*\mathbb{T}M$.
  \end{prop}
  \begin{proof}
  It is well known that $$\norm{\pi_t(f)}\leq\sup_{x\in M}\max\left(t^{-\dim(M)}\int_M |f(x,y,t)|dy,t^{-\dim(M)}\int_M |f(y,x,t)|dy\right) ,\ t\in \R_+^\times,f\in C^\infty_c(\mathbb{T}M)$$
  and $$\norm{\pi_{x}(f)}\leq\sup_{x\in M}\left(\int_{T_xM}|f(v,x,0)|dv\right) ,\quad x\in M,f\in C^\infty_c(\mathbb{T}M).$$
   Hence, \begin{equation}\label{eqn:qsijfiqjmsdjofkjqskodfjm}
  \norm{f}\leq \sup_{x\in M,t\in \R_+^\times}\max\left(t^{-\dim(M)}\int_M |f(x,y,t)|dy,t^{-\dim(M)}\int |f(y,x,t)|dy,\int_{T_xM}|f(v,x,0)|dv\right)
  \end{equation}
  It follows that $C^*\mathbb{T}M$ contains measurable functions on $\mathbb{T}M$ for which the right-hand side of \eqref{eqn:qsijfiqjmsdjofkjqskodfjm} is finite (usually denoted $L^1(\mathbb{T}M)$). Let $\phi$ be as in \eqref{eqn:phi}. If $f\in \cS(\mathbb{T}M)$, then the right-hand side of \eqref{eqn:qsijfiqjmsdjofkjqskodfjm} is bounded by $\norm{\phi f}_{\infty}$. Hence, $f\in  C^*\mathbb{T}M$. Finally, $\cS(\mathbb{T}M)$ is a $*$-subalgebra by Proposition \ref{prop:mainprop_Schwartz}.
  \end{proof}
  \paragraph{Relation between uniform norm and $C^*$-norm.}
  The following theorem which is essentially just the Sobolev embedding theorem will be very useful in allowing us to replace the uniform norm with the $C^*$-norm, which is more convenient to use.
  \begin{theorem}\label{thm:twonorms_COnnes}
  Let $\Delta$ be the positive Laplace-Beltrami operator on $M$, $\mathbblD$ the corresponding differential operator on $\mathbb{T}M$ as in \eqref{eqn:mathbbDConnes}, and $k\in \mathbb{N}$ with $2k>\frac{\dim(M)}{2}$. There exists $C>0$, such that $$\norm{f}_{\infty}\leq C \norm{  (1+t^{\dim(M)})(1+\mathbblD)^{k}\star f\star  (1+ \mathbblD)^{k}},\quad \forall f\in C^\infty_c(\mathbb{T}M)$$
  \end{theorem}
  \begin{lem}Let $k\in \N$ with $2k>\frac{\dim(M)}{2}$. There exists a constant $C>0$ such that for all $t>0$ and $x\in M$, if $u_x$ denotes the Dirac delta distribution on $M$ at $x$, then $$\norm{(1+t^2\Delta)^{-k}(u_x)}_{L^2M}\leq C\max(t^{-\frac{\dim(M)}{2}},1).$$
  \end{lem}
  \begin{proof}
  The Sobolev embedding lemma implies that $u_x\in H^{-2k}(M)$ where $H^s(M)$ denotes the $s$ Sobolev space. Hence, $(1+t^2\Delta)^{-k}(u_x)\in L^2M$. For $t>1$ we have and use the inequality \begin{align*}
  \norm{(1+t^2\Delta)^{-k}(u_x)}_{L^2M}&= \norm{(1+t^{2}\Delta)^{-k}(1+\Delta)^{k}(1+\Delta)^{-k}(u_x)}_{L^2M}\\&\leq \norm{(1+t^{2}\Delta)^{-k}(1+\Delta)^{k}}_{B(L^2M)}\norm{(1+\Delta)^{-k}(u_x)}_{L^2M}\\&\leq \norm{(1+\Delta)^{-k}(u_x)}_{L^2M},
  \end{align*}
  where in the last inequality we used the fact that $\frac{(1+x)^k}{(1+t^2x)^k}\leq 1$ for all $x\in [0,+\infty[$. For $t<1$, we proceed differently. If $g\in L^2M$, then $$\langle (1+t^2\Delta)^{-k}(u_x),g\rangle_{L^2M}=|(1+t^2\Delta)^{-k}(g)(x)|.$$ 
  We need to maximize $\frac{|(1+t^2\Delta)^{-k}(g)(x)|}{\norm{g}_{L^2M}}$ as $g$ varies in $L^2M$. The function $g$ needs to vanish outside any given neighbourhood of $x$ to reach the supremum. We can thus assume that $M=\R^{\dim(M)}$ and $x=0$ and the metric on $M$ is Euclidean outside the unit ball. The inequality follows from a change of variable $x\to \frac{x}{t}$.
  %To prove the bound, consider the map $$r:\mathbb{T}M\to M\times \R_+,\quad (x,y,t)\mapsto (x,t),\quad (x,v,0)\mapsto (x,0).$$ The map $r$ is a submersion. We now think of $\mathbb{T}M$ as a family of smooth manifolds $r^{-1}(x,t)$ as $(x,t)$ varies in $M\times \R_+$. Let $g$ denote the Riemannian metric on $M$. We equip the manifold $r^{-1}(x,t)$ with the Riemannian metric $\frac{g}{t^2}$ for $t\neq 0$ and with the constant Riemannian metric $g_x$ for $t=0$. One can check that this gives a smooth Euclidean metric on $\ker(dr)$. Let $u_{x,t}$ be the Dirac delta distribution at $(x,x,t)\in r^{-1}(x,t)$. We now apply the Sobolev embedding lemma, around $u_{(x,0)}$ to deduce that if $\tilde{\Delta}$ denotes the Laplace operator on $r^{-1}(x,t)$. Then there exists $C>0$ such that $$\norm{(1+\tilde{\Delta})^{-k}(u_{x,t})}_{L^2r^{-1}(x,t)}\leq C$$ for all $t$ small enough. But $r^{-1}(x,t)$ is just $M$ with a factor of $\frac{1}{t^2}$ in the metric. One thus deduces that $$\norm{(1+t^2\Delta)^{-k}(u_{x})}_{L^2M}\leq Ct^{\dim(M)/2}$$ for $t$ small enough. Since $M$ is compact, the above bound can be taken uniformly in $x$. This proves the lemma for $t$ small. 
  \end{proof}
  \begin{proof}[Proof of Theorem \ref{thm:twonorms_COnnes}]Let $t>0$. We apply $$\pi_t((1+\mathbblD)^{k}\star f\star  (1+ \mathbblD)^{k})=t^{-\dim(M)}(1+t^2\Delta)^{k}\star f\star  (1+ t^2\Delta)^{k}$$ to $(1+t^2\Delta)^{-k}(u_x)$ to deduce that 
  \begin{align*}
   t^{-\dim(M)}\norm{(1+t^2\Delta)^{k}( f(\cdot,x,t))}_{L^2M}&\leq C\max(t^{-\frac{\dim(M)}{2}},1)\norm{  (1+\mathbblD)^{k}\star f\star  (1+ \mathbblD)^{k}}
   \end{align*}
  One has \begin{align*}
  |f(y,x,t)|&=|\left\langle u_y,f(\cdot,x,t)\right\rangle_{H^{-2k}(M)\times H^{2k}(M)}|\\&=|\langle (1+t^2\Delta)^{-k}(u_y),(1+t^2\Delta)^{k}(f(\cdot,x,t))\rangle_{L^2(M)\times L^2(M)}|\\&\leq \norm{(1+t^2\Delta)^{-k}(u_y)}_{L^2M}\norm{(1+t^2\Delta)^{k}( f(\cdot,x,t))}_{L^2M}\\	&\leq C\max(t^{-\frac{\dim(M)}{2}},1)\norm{(1+t^2\Delta)^{k}( f(\cdot,x,t))}_{L^2M}
  \end{align*}
  It follows that \begin{align*}
  |f(y,x,t)|&\leq C^2t^{\dim(M)}\max(t^{-\dim(M)},1)\norm{  (1+\mathbblD)^{k}\star f\star  (1+ \mathbblD)^{k}}\\&=C^2\max(1,t^{\dim(M)})\norm{(1+\mathbblD)^{k}\star f\star  (1+ \mathbblD)^{k}}.
  \end{align*}
  By taking the limit as $t\to 0^+$, one obtains a bound of $f$ on $TM\times\{0\}$. The theorem follows.
  \end{proof}
  
  \paragraph{Schwartz functions are closed under smooth calculus.}
  \begin{proof}[Proof of Theorem \ref{thm:intro:Schwartz_closed_holo_Connes}] 
    Let $\cD\subseteq \End(\cS(\mathbb{T}M))$ be the subalgebra of $\End(\cS(\mathbb{T}M))$ generated by the maps in Definition \ref{dfn:SchwartzConnes}. 
    We will prove that $\cD(\cS(\mathbb{T}M))=\cS(\mathbb{T}M)$. Theorem \ref{thm:intro:Schwartz_closed_holo_Connes} would then follow from Theorem \ref{thm:funccalc_high_der}. 
    Let $a\in \cD(\cS(\mathbb{T}M))$, $x_n\in \cS(\mathbb{T}M)$ as in Definition \ref{dfn:cDcA}. 
    Since $$x\mapsto (1+t^{\dim(M)})(1+\mathbblD)^{k}\star x\star  (1+ \mathbblD)^{k}$$ is a differential operator on $\cS(\mathbb{T}M)$ which belongs to $\cD$, it follows that $ (1+t^{\dim(M)})(1+\mathbblD)^{k}\star x_n\star  (1+ \mathbblD)^{k}$
     converges in $C^*\mathbb{T}M$. By Theorem \ref{thm:twonorms_COnnes}, it follows that $a$ is a bounded continuous function and $x_n\to a$ in uniform norm.
   Let $\delta\in \cD$. Since the map $x\mapsto (1+t^{\dim(M)})(1+\mathbblD)^{k}\star \delta(x)\star  (1+ \mathbblD)^{k}$ belongs to $\cD$, again by Theorem \ref{thm:twonorms_COnnes}, $\delta(x_n)$ converges in the uniform norm to a bounded continuous function on $\mathbb{T}M$. It follows that $a$ is smooth and $\delta(a)$ bounded for every $\delta$. Hence, $a\in \cS(\mathbb{T}M)$.
  \end{proof}
  \begin{rem}
    In \cite{EwertSchwartz}, Ewert defined an algebra of Schwartz functions for the inhomogeneous tangent groupoid defined in \cite{ErikBobTangentGrp,MohsenGrpTangent,HigsonHaj,ChoiPonge}. 
    Does Theorem \ref{thm:intro:Schwartz_closed_holo_Connes} hold for this Schwartz algebra, or a similarly defined one?
  \end{rem}
\begin{refcontext}[sorting=nyt]
\printbibliography
\end{refcontext}
{\footnotesize
(Omar Mohsen) Paris-Saclay University, Paris, France
\vskip-2pt e-mail: \texttt{omar.mohsen@universite-paris-saclay.fr}}
\end{document}